\documentclass[11pt,reqno]{amsart}
\usepackage[margin=1in]{geometry}
\usepackage{amssymb}
\usepackage{amsmath}
\usepackage{mathtools}
\usepackage{latexsym}
\usepackage{appendix}
\usepackage{mathrsfs}
\usepackage{amsthm,stmaryrd}
\usepackage{graphicx}
\usepackage{caption}
\usepackage{dsfont}
\usepackage{hyperref}
\hypersetup{hidelinks}
\usepackage{enumerate}
\usepackage{wasysym}
\usepackage{cancel}
\usepackage{color} 
\usepackage{comment}
\usepackage{xcolor}
\usepackage{bigints}
\usepackage{url}
\usepackage{bm}
\usepackage{float}
\usepackage[dvipsnames]{xcolor}
\usepackage{needspace}
\usepackage{tikz}
\usetikzlibrary{cd}
\usepackage[utf8]{inputenc}
\usepackage[T1]{fontenc} 
\usepackage{adjustbox}
\usepackage{graphicx}
\usepackage{graphicx,tipa}

\numberwithin{equation}{section}
\theoremstyle{plain}
\newtheorem{thm}{Theorem}[section]  
\newtheorem{prop}[thm]{Proposition} 
\newtheorem{lem}[thm]{Lemma}

\newtheorem{hyp}{Hypothesis}

\theoremstyle{definition}
\newtheorem{remark}[thm]{Remark}

\newcommand{\lap}{\Delta}
\newcommand{\grad}{\nabla}
\newcommand{\eps}{{\epsilon}}
\newcommand{\vareps}{{\varepsilon}}
\newcommand{\df}{\coloneqq}
\newcommand{\eqn}[1]{\begin{equation*} #1 \end{equation*}}
\newcommand{\eqnum}[1]{\begin{equation} #1 \end{equation}}
\newcommand{\noi}{\noindent}
\newcommand{\parnoi}{\par \noi}
\newcommand{\RR}{\mathbb{R}}
\newcommand{\NN}{\mathbb{N}}
\newcommand{\BB}{\mathbb{B}}
\newcommand{\lm}{\lambda}
\newcommand{\om}{\omega}
\newcommand{\Om}{\Omega}
\newcommand{\Snm}{{\mathbb{S}^{n-1}}}
\newcommand{\SSS}{\mathbb{S}}
\DeclareMathOperator{\dv}{d\hspace{-0.1em}V}
\newcommand{\half}{\frac{1}{2}}
\newcommand{\p}{^{(p)}}

\newcommand{\q}{^{(q)}}
\newcommand{\qq}{^{(q+1)}}
\DeclareMathAlphabet{\mathpzc}{OT1}{pzc}{m}{it}
\newcommand{\normal}{{\mathpzc{n}}}
\DeclareMathOperator{\ex}{\mathrm{ex}}
\DeclareMathOperator{\co}{\mathrm{co}}
\newcommand{\abseig}{\mu\p}
\newcommand{\releig}{\lm\p}
\newcommand{\exeig}[2]{\mu^{{(#1)}}_{{#2,\ex}}}
\newcommand{\exeigpo}{\exeig{p}{1}}
\newcommand{\coeig}[2]{\mu^{(#1)}_{#2,\co}}
\newcommand{\coeigpo}{\coeig{p}{1}}
\newcommand{\Rh}{R_h}
\newcommand{\Rhh}{\hat{R}_h}
\newcommand{\rhh}{\hat{r}_h}
\newcommand{\BnRh}{\BB_{\Rh}}
\newcommand{\SRh}{\SSS_{\Rh}}
\newcommand{\dom}{\mathfrak{D}}
\newcommand{\Dom}{{\widetilde\dom}}
\newcommand{\pa}{\partial}
\newcommand{\calB}{\mathcal{B}}
\newcommand{\RhRc}{{\frac{\Rh}{R_c}}}
\newcommand{\Rhrc}{\frac{\Rh}{r_0}}
\DeclareMathOperator{\vol}{\mathrm{Vol}}
\newcommand{\proj}{\mathpzc{P}}
\DeclareMathOperator{\card}{\mathrm{Card}}
\newcommand{\calO}{\mathcal{O}}
\DeclareMathOperator{\jac}{\mathrm{Jac}}
\DeclareMathOperator{\id}{\mathpzc{I}}
\DeclareMathOperator{\op}{op}
\newcommand{\inv}{^{-1}}
\DeclareMathOperator{\diam}{\mathrm{Diam}}
\DeclareMathOperator{\dist}{\mathrm{Dist}}
\newcommand{\omnorm}{\Vert\om\Vert_{L^2\Om^p(\dom)}}
\newcommand{\ann}{\mathcal{A}}
\newcommand{\SSp}{{\mathbb{S}^p}}
\newcommand{\celleps}[1]{{\widetilde{\mathcal{C}}_{#1}}}
\newcommand{\thetatil}{{\tilde{\theta}}}
\newcommand{\Appr}[1]{O\left(#1\right)}
\DeclareMathOperator{\supp}{Supp}
\newcommand{\holes}{{\mathfrak{h}}}
\newcommand{\Inv}[1]{{\frac{1}{#1}}}
\DeclareMathOperator{\intr}{Int}
\newcommand{\restr}{\mathpzc{r}}
\newcommand{\diff}{\hat{\delta}}
\newcommand{\nuh}{\hat{\nu}}
\newcommand{\etah}{\hat{\eta}}
\newcommand{\etab}{\bar{\eta}}
\newcommand{\psib}{\bar{\psi}}
\newcommand{\nrm}[1]{\left\Vert #1 \right\Vert}
\newcommand{\Rpart}{R_{\scriptscriptstyle \hspace{-0.1em}\mathcal{P}}}
\newcommand{\Rch}{\hat{R}_c}
\newcommand{\inn}[3]{\langle\hspace{-3pt}\langle {#2},{#3}\rangle\hspace{-3pt}\rangle_{#1}}
\DeclareMathOperator{\Rel}{rel}
\newcommand{\rel}{{\Rel}}
\DeclareMathOperator{\cl}{cl}
\DeclareMathOperator{\cc}{cc}

\allowdisplaybreaks
\begin{document}

\title[]{Lower bounds for the eigenvalues of the Hodge Laplacian on certain non-convex domains}

\author{Tirumala Chakradhar}
\address{Tirumala Chakradhar (corresponding author), University of Bristol, School of Mathematics, Fry Building, Woodland Road, Bristol BS8 1UG, United Kingdom.}
\email{tirumala.chakradhar@bristol.ac.uk}

\author{Pierre Nicolle-Guerini}
\address{Pierre Nicolle-Guerini, Ecole Sp\'eciale Militaire de Saint-Cyr, D\'epartement de Math\'ematiques \& Informatique, 240, Avenue de l'Ecole Sp\'eciale Militaire, 78211 Saint-Cyr-l'Ecole, France.}
\email{pguerini@ac-versailles.fr}

\begin{abstract}
We establish geometric lower bounds for the smallest positive eigenvalue of the Hodge Laplacian in the class of non-convex domains given by Euclidean annular regions with a convex outer boundary and a spherical inner boundary. These bounds are then extended to convex domains with multiple holes, where we derive lower bounds for certain higher order exact eigenvalues, and under additional geometric assumptions, also for the smallest positive eigenvalue. For $1$\nobreakdash-forms on compact manifolds with boundary, we provide a general lower bound on the smallest exact eigenvalue\enspace--\enspace corresponding to the first positive Neumann eigenvalue\enspace--\enspace which, in certain respects, is better than the classical Cheeger inequality. Furthermore, we emphasise the necessity of the \emph{contact radius} in the lower bounds of the main results. Our proofs employ local-to-global arguments via an explicit isomorphism between \v{C}ech cohomology and de~Rham cohomology to obtain Poincar\'e-type inequalities with explicit geometric dependence, and utilise certain generalised versions of the Cheeger--McGowan gluing lemma.
\end{abstract}


\maketitle
\tableofcontents

\section{Introduction} \label{sec: intro}
The Hodge Laplacian on a Riemannian $n$\nobreakdash-manifold $(M^n,g)$, that acts on the space of smooth differential $p$\nobreakdash-forms on $M$, for $p\in \{0,\dots,n\}$, is defined to be
\eqn{
\lap\p\df d\delta+ \delta d,
}
where $d$ denotes the exterior derivative and $\delta$ the codifferential. It is a generalisation of the Laplace-Beltrami operator acting on smooth functions on $M$ (i.e., $0$\nobreakdash-forms), and is well known for its interplay with geometry, analysis, and topology, besides its significance in modern mathematical physics. The spectral geometry of the Hodge Laplacian for $p>0$ is often more sophisticated and interesting compared to the Laplacian on functions, thanks to a greater interaction with the underlying topology arising from Hodge theory. For instance, the analogue of the classical Cheeger inequality which gives a lower bound for the spectrum in the case of functions, is not known in general for $p$\nobreakdash-forms. See also the recent work~\cite{BC22} for a Cheeger-like lower bound on the smallest coexact eigenvalue for $1$\nobreakdash-forms, which depends on certain curvature bounds. 

\par
Eigenvalue bounds for the Hodge Laplacian have been extensively studied in the literature. See \cite{GM75,dod82,CT97,Gue04ii,tak05b,sav09,RS11,sav14,kwo16,CHMN18,CS19,EGH21} for some interesting lower bounds based on geometric and topological assumptions such as restrictions on the curvature, cohomology, codimension of isometric immersions, properties of certain `good covers', and, in the case of manifolds with boundary, additional assumptions such as the star convexity of the domain or the $p$\nobreakdash-convexity of the boundary. These type of bounds often employ some versions of the Bochner-Weitzenb\"ock identity or the Reilly formula for differential forms. On the other hand, articles such as \cite{CC90,PS01,tak03,lot04,jam11,AT24} describe the construction of manifolds with small eigenvalues, the prescription of the $p$\nobreakdash-spectrum under various restrictions, and the characterisation of manifolds with small eigenvalues. 
\par
In this article, we give a lower bound on the smallest positive eigenvalue of the Hodge Laplacian on $p$\nobreakdash-forms with absolute boundary conditions, for a Euclidean annular domain whose outer boundary is convex and the inner boundary is a round sphere, and study its extension to a family of perforated domains with convex outer boundaries.
\parnoi
For convex Euclidean domains $\mathcal{D}\subset \RR^n$, it was shown by the second named author~\cite{Gue04i} that the smallest positive absolute eigenvalue is bounded below by
\eqn{
\abseig_1 (\mathcal{D}) \geq \frac{C_1}{\diam(\mathcal{D})^2}, \qquad \forall \, p \in \{0,\dots,n\},
}
where $C_1$ is an explicit constant depending only on the dimension $n$ and the degree $p$ of the differential forms. A more refined version of these bounds were obtained by Savo~\cite{sav11}, also for convex Euclidean domains:
\eqn{
\frac{C_2}{D_p(\mathcal{D})^2}\leq \abseig_1(\mathcal{D})\leq \frac{C_3}{D_p(\mathcal{D})^2},\qquad \forall \, p \in \{1,\dots,n\}, 
}
where $C_2$ and $C_3$ are explicit constants depending only on $n$ and $p$; $D_p(\mathcal{D})$ is the length of the $p^\text{th}$ longest principal axis of the largest ellipsoid that can be inscribed in $\mathcal{D}$, known as the John ellipsoid. For example, $D_1(\mathcal{D})$ corresponds to the diameter, and $D_n(\mathcal{D})$ corresponds to the inradius of $\mathcal{D}$. These bounds were also generalised for sufficiently small geodesically convex domains in manifolds with non-negative curvature operator~\cite{KK20}.
\parnoi
Furthermore, for $p\in\{0,\dots,n-2\}$, Guerini~\cite[Theorem~2.1]{Gue04ii} constructed non-convex domains diffeomorphic to balls, with bounded diameters, which have an arbitrarily small positive eigenvalue for $p$\nobreakdash-forms. The convexity assumption in \cite{Gue04i} is a means to avoid the creation of non-trivial cohomology through collapsing as done in \cite{Gue04ii} by gluing a thick and a thin part that generalise the classical Cheeger dumbbell.
\parnoi
The lower bounds of this article for Euclidean annular regions depend on the distance between the inner hole and the outer boundary, which we refer to as the  \emph{contact radius} $R_c$.
\begin{figure}[H]
    \centering
    \includegraphics[width=0.35\linewidth]{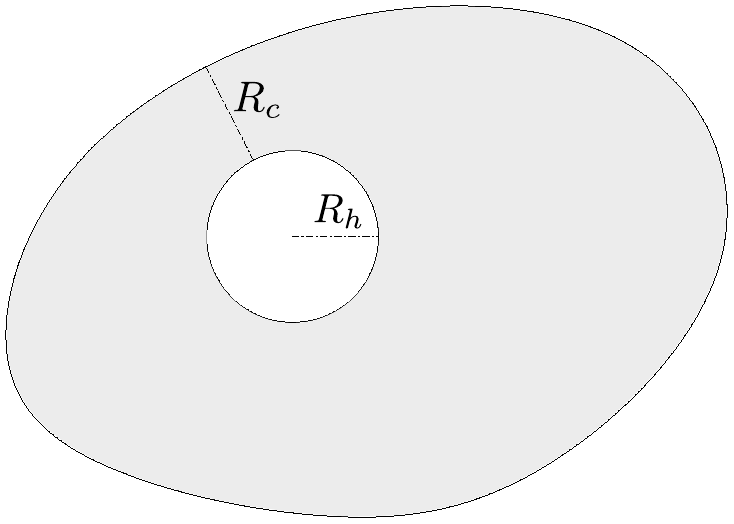}
    \caption{}
\end{figure}
\parnoi
\textbf{Notation. }Throughout this article, $K$ denotes a generic constant depending only on the dimension $n$, possibly different in different expressions. We use $K_i$ to denote specific constants, also depending only on the dimension $n$.
\parnoi
We now state the main results.
\begin{thm}\label{thm: first abs eigval lower bds for forms on annular} Let $\dom\subset \RR^n$ be a Euclidean annular domain, given by a smooth, bounded convex domain minus a hole. Then the following lower bounds hold for the smallest exact eigenvalue of the Hodge Laplacian on $p$\nobreakdash-forms.
\begin{enumerate}
\item For $n\geq 3$ and $p\in \{2,\dots,n-1\}$,
\eqn{
\exeigpo(\dom)\geq K \frac{1}{D^2}\left(\frac{R_c}{D} \right)^{5n^2-n-5}.
}
\item For $n\geq 2$ and $p=1$,
\eqn{
\exeig{1}{1}(\dom)\geq K \frac{1}{D^2}\left(\frac{R_c}{D}\right)^{11n-5}.
}
\end{enumerate}
Here, $D$ and $R_c$ denote the diameter and the contact radius of $\dom$, respectively.
\end{thm}
\vspace{0em}
\begin{remark}
The above lower bounds on $\exeigpo$ for each $p\in \{1,\dots, n-1\}$ will suffice to deduce lower bounds for the smallest positive absolute eigenvalue for all $p\in \{0,\dots,n\}$: For $p=0$, the smallest positive Neumann eigenvalue $\mu_1=\exeig{1}{1}$; for $p\in \{1,\dots,n-2\}$, we have
\eqn{
\abseig_1=\min\{\exeigpo,\coeigpo\}=\min\{\exeigpo,\exeig{p+1}{1}\}
}
in view of the correspondence between the coexact $p$\nobreakdash-spectrum and the exact $(p+1)$\nobreakdash-spectrum. Moreover, we have from the duality between the absolute and the relative spectra that
\eqn{
\mu^{(n)}_1=\lm_1,
}
where $\lm_1$ is the first Dirichlet eigenvalue for functions, and
\eqn{
\mu^{(n-1)}_1=\min\{\exeig{n-1}{1},\coeig{n-1}{1}\}=\min\{\exeig{n-1}{1},\exeig{n}{1}\}=\min\{\exeig{n-1}{1},\mu^{(n)}_1\}=\min\{\exeig{n-1}{1},\lm_1\},
}
which covers the lower bounds on $\mu^{(n-1)}_1$ and $\mu^{(n)}_1$ using the Faber-Krahn inequality
\eqn{
\lm_1(\dom)\geq \frac{K}{\diam(\dom)^2}.
}
In particular, we have new lower bounds for the Neumann eigenvalues of annular domains. Note that, for planar domains, the spectrum for $1$\nobreakdash-forms consists solely of the Neumann and the Dirichlet eigenvalues for functions, and understanding these would already give a complete picture of the full spectrum of the Hodge Laplacian for $p\in\{0,1,2\}$.
\end{remark}
\parnoi
To prove Theorem~\ref{thm: first abs eigval lower bds for forms on annular}, we adapt the methods of Chanillo and Treves~\cite{CT97}\enspace--\enspace employing tools from \v{C}ech--de Rham theory\enspace--\enspace in the scenario where the inner hole is sufficiently large compared to the contact radius; while for the case where the hole is much smaller than the contact radius, we use a combination of the bounds from the former case and the spectral convergence of Ann\'e and Colbois~\cite{AC93}, with the help of a useful gluing lemma by McGowan~\cite{mcg93}. See Section~\ref{sec: main lower bds}.
\par
We also prove the following theorem, which gives a better lower bound for the smallest positive Neumann eigenvalue than the one deduced from the Cheeger inequality. This theorem would be of use in deriving the lower bounds of Theorem~\ref{thm: first abs eigval lower bds for forms on annular} for the case of functions and $1$\nobreakdash-forms.

\begin{thm}\label{thm: neumann lower bd for func in terms of union}
Let $(M^n,g)$ be a compact, connected, smooth Riemannian manifold with boundary ($n\geq 2$) and $\{U_1,U_2\}$ be an open cover with $U_1$ and $U_2$ (each of which is connected) having smooth boundaries. Then the first positive Neumann eigenvalue satisfies
\eqn{
\mu_1(M)\geq \frac{1}{32} \frac{\vol(U_1\cap U_2)}{\vol(M)}\min \left\{\mu_1(U_1),\mu_1(U_2)\right\}.
}
\end{thm}
\vspace{0em}
\begin{remark}
The theorem also holds for closed  manifolds, with $\mu_1(M)$ on the LHS replaced by the first positive eigenvalue of the Laplace-Beltrami operator on $M$.
\end{remark}
\noi
The above inequality is geometrically optimal in the following sense: It is known from \cite{ann95} that the first positive Neumann eigenvalue on a Euclidean Cheeger dumbbell with a cylinder of radius $\eps$ tends to $0$ at a rate of $\eps^{n-1}$, which is same as the rate we get by applying our inequality with the open sets $U_1$ and $U_2$ each chosen to be the union of a ball and a cylinder of radius $\eps$, because the ratio $\frac{\vol(U_1\cap U_2)}{\vol(M)}$ too goes as $\eps^{n-1}$ and the eigenvalues $\mu_1$ of $U_1$ and $U_2$ are known to be bounded below independent of $\eps$, from \cite{Gue04ii}. On the other hand, recall that the Cheeger inequality gives a rate of $\eps^{2(n-1)}$.
\begin{figure}[H]
    \centering
    \includegraphics[width=0.4\linewidth]{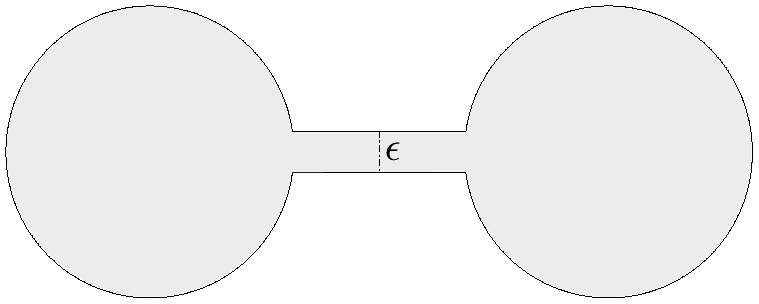}
    \caption{}
\end{figure}
\parnoi
Note, however, that the second positive Neumann eigenvalue of a manifold satisfies the lower bound~\cite[Corollary~4.65]{text-GHL}
\eqn{
\mu_2(M)\geq \half \min\left\{{\mu_1}(U_1),{\mu_1}(U_2)\right\},
}
and one would require a choice of $U_1$ and $U_2$ with small first positive eigenvalues in order to construct a manifold with small second positive eigenvalue.
\parnoi
Theorem~\ref{thm: neumann lower bd for func in terms of union} is proved in the latter half of Section~\ref{sec: main lower bds}, using only the variational characterisation for the Neumann eigenvalue and careful rearrangements of the terms involved.
\par
Next, we show that the dependence of the lower bounds of Theorem~\ref{thm: first abs eigval lower bds for forms on annular} on the contact radius is necessary.
\begin{thm}\label{thm: domains with small eigval}
Let $n\geq 2$ and $p\in \{1,\dots,n-1\}$. There exists a family $(\ann_\eps^p)_{\eps\in(0,\half)}$ of Euclidean annular $C^1$\nobreakdash-domains with convex outer boundary, such that $\lim_{\eps\to 0}\exeigpo(\ann_\eps^p)=0$, with the diameter uniformly bounded and the contact radius tending to zero. It is also implied that the result holds for the smallest positive Neumann eigenvalue for functions, since $\mu_1=\exeig{1}{1}$.
\end{thm}
\noi
The main idea of Theorem~\ref{thm: domains with small eigval} involves the construction of an appropriate family of domains that tend to a limiting domain whose $p^\text{th}$ cohomology has a higher dimension. Then, using the harmonic form associated with the extra dimension of the cohomology on the limiting domain, we construct suitable test forms on these domains, to be used in the variational characterisation for the eigenvalues.
\par
We now extend the lower bounds to the exact eigenvalues for convex Euclidean domains with multiple holes. We consider smooth, bounded domains $\dom\subset \RR^n$ obtained by removing $\holes$ many non-overlapping balls (away from the boundary), of radius $\Rh$ each, from a convex domain. Let $D$ denote the diameter of $\dom$, and define the \textit{contact radius} of such domains with multiple holes to be
\eqn{
\hat{R}_c(\dom)\df \min \left\{r_c,\frac{d_h}{2}\right\},
}
where $r_c$ is the distance between the outer boundary and a hole closest to it, and $d_h$ is the distance between a pair of holes closest to each other.
\begin{thm}\label{thm: induction method bounds for identical holes}
Let $\dom\subset \RR^n$ be a smooth, bounded domain as above, with convex outer boundary and $\holes$ many holes of radius $\Rh$ each. Suppose that the Voronoi partition of $\dom$ with respect to the centers of the holes, up to reordering, is such that the union of the first $\ell$ pieces has a convex outer boundary, for every $\ell\in\{1,\dots,\holes\}$. Then we have the following lower bounds for the first exact eigenavalue.
\begin{enumerate}[(i)]
\item For $n\geq 3$ and $p\in \{2,\dots,n-1\}$,
\eqn{
\exeigpo(\dom)\geq K \frac{1}{D^2} \left(\frac{\hat{R}_c}{D}\right)^{5n^2-n-7+2\holes}.
}
\item For $n\geq 2$ and $p=1$,
\eqn{
\exeig{1}{1}(\dom)\geq K \frac{1}{D^2} \left(\frac{\hat{R}_c\Rh^{n-1}}{D^n}\right)^{\holes-1}\left(\frac{\hat{R}_c}{D}\right)^{11n-5}.
}
\end{enumerate}
\end{thm}
\noi
A simple example of a domain satisfying the above hypothesis is as follows.
\begin{figure}[H]
    \centering
    \includegraphics[width=0.75\linewidth]{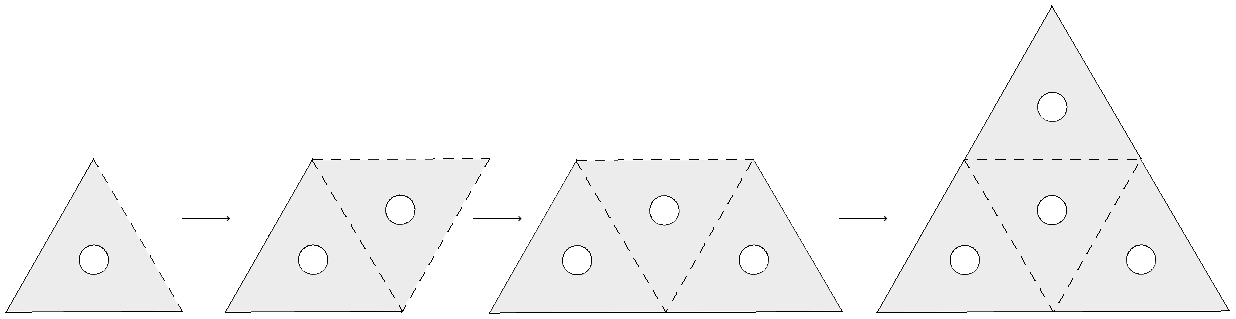}
    \caption{Example I}
\end{figure}
\noi
Although we use the Voronoi partition in the above theorem so as to have $\Rch$ defined independent of the partition, with an optimal contribution to the bounds, the lower bounds also hold if we consider any partition of $\dom$ into $\holes$ many pieces, each of which is a convex domain with a hole in the interior, and modify the definition of $\Rch$ appropriately. This also allows, for example, domains of the following type whose Voronoi partitions themselves may not satisfy our hypothesis unlike Example~I.
\begin{figure}[H]
    \centering
    \includegraphics[width=0.70\linewidth]{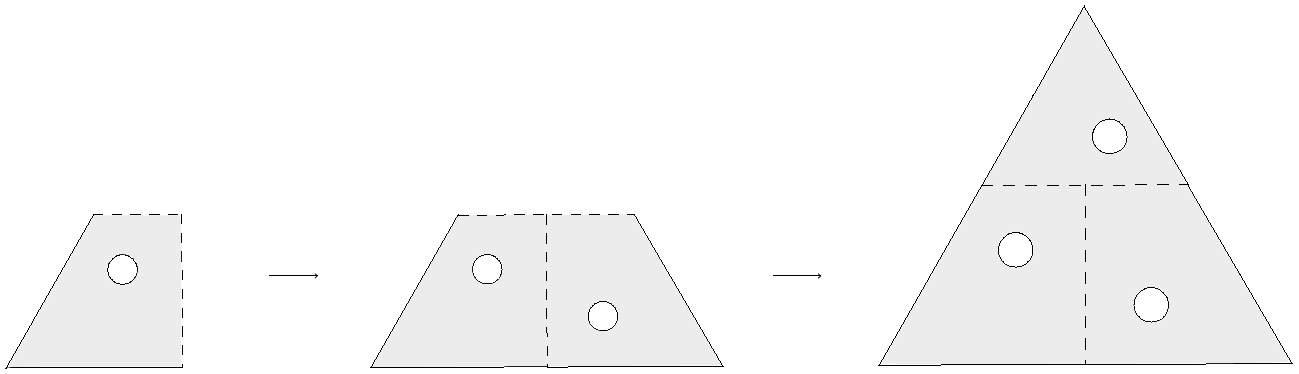}
    \caption{Example I\!I}
\end{figure}
\noi
Further, similar lower bounds can be obtained for a larger class of domains than those satisfying the above versions of the geometric assumption. See Remark~\ref{rmk: weaker hyp for induction method multi holes}. This allows in addition, the following domain, for instance, which is a union of the above two examples (see also Figure~\ref{fig:ExampleIV}).
\begin{figure}[H]
    \centering
    \includegraphics[width=0.35\linewidth]{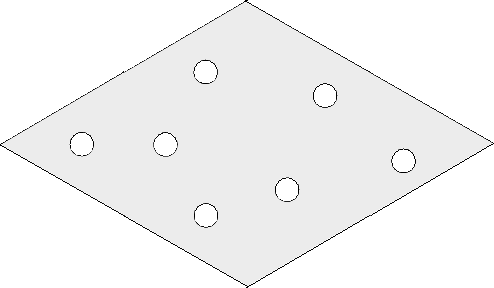}
    \caption{Example I\!I\!I}
    \label{fig:ExampleIII}
\end{figure}
\noi
The following, however, is a non-example. For this domain, no partition with convex outer boundaries satisfies our hypothesis, nor can we apply the extension provided by Remark~\ref{rmk: weaker hyp for induction method multi holes}.
\begin{figure}[H]
    \centering
    \includegraphics[width=0.25\linewidth]{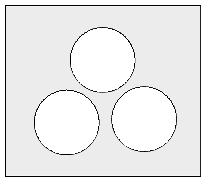}
    \caption{Non-example}
\end{figure}
\begin{remark}
We have stated Theorem~\ref{thm: induction method bounds for identical holes} for convex domains with holes of identical size, for simplicity. Analogous lower bounds in the generic setting of convex domains with multiple holes of varying size are presented in Theorem~\ref{thm: induction method bounds for multi holes}. For $p\in \{2,\dots,n-1\}$, we further highlight the necessity of the partition parameter $\Rpart$ (see Section~\ref{sec: lower bds multiple holes}) in the lower bounds by constructing a family of convex domains with multiple holes (with the diameter of the domain), satisfying the hypothesis of Theorem~\ref{thm: induction method bounds for multi holes}, having arbitrarily small eigenvalues as $\Rpart\to 0$. See Theorem~\ref{thm: dom with small eigval multi holes}.
\end{remark}
\par
The geometric assumption in Theorem~\ref{thm: induction method bounds for identical holes} (and Theorem~\ref{thm: induction method bounds for multi holes}) may be removed, with the price that we obtain lower bounds for higher order exact eigenvalues instead of the smallest one. These lower bounds are later presented in Theorem~\ref{thm: mcgowan method bounds for multi holes} for convex domains with possibly non-identical holes. The exponent of the ratio $\hat{R}_c/D$ in these bounds is independent of the number of holes, in contrast to the expressions of Theorems~\ref{thm: induction method bounds for identical holes}, and instead, a factor of $\holes^{2p-1}$ appears in the denominator of the lower bounds.
\par
The proof of Theorem~\ref{thm: induction method bounds for identical holes} (resp. Theorem~\ref{thm: induction method bounds for multi holes}) is based on an induction approach where we consider the sequence of annular domains with convex outer boundary obtained from the Voronoi partition (resp. Laguerre-Voronoi partition) of $\dom$ given by the hypothesis and use an adaptation of the gluing lemmas of McGowan~\cite[Lemma~2.3]{mcg93} and Gentile-Pagliara~\cite[Lemma~1]{GP95} to recursively employ the lower bounds of Theorem~\ref{thm: first abs eigval lower bds for forms on annular} for the eigenvalues of the pieces and the earlier mentioned bounds of~\cite{Gue04i} for the convex intersections of their open neighborhoods constructed suitably. On the other hand, we prove Theorem~\ref{thm: mcgowan method bounds for multi holes} by generalising McGowan's lemma to higher degree forms under appropriate topological assumptions on the open cover and applying it to the cover of the domain $\dom$ obtained from its Laguerre-Voronoi partition, in combination with the bounds mentioned above. See Section~\ref{sec: lower bds multiple holes} for the details.
\par
This article is organised as follows. In Section~\ref{sec: prelim} on preliminaries, we introduce the absolute and the relative eigenvalue problems for the Hodge Laplacian on compact Riemannian manifolds with boundary, followed by a brief description of the standard spectral properties and a version of the Hodge decomposition that is of relevance to us. In Section~\ref{sec: main lower bds}, we prove the lower bounds of Theorem~\ref{thm: first abs eigval lower bds for forms on annular} for the first exact eigenvalue of annular domains, and Theorem~\ref{thm: neumann lower bd for func in terms of union} for the lower bounds on the Neumann eigenvalues of a manifold in terms of those that of the covering subdomains. In Section~\ref{sec: small eigvals}, we prove Theorem~\ref{thm: domains with small eigval} by constructing annular domains with arbitrarily small eigenvalues and emphasise the necessity of the contact radius in our lower bounds. In Section~\ref{sec: lower bds multiple holes}, we describe the Laguerre-Voronoi partition for domains with multiple holes, then state and prove Theorem~\ref{thm: mcgowan method bounds for multi holes} about the lower bounds for higher order exact eigenvalues of domains with multiple holes, and Theorem~\ref{thm: induction method bounds for multi holes} that generalises Theorem~\ref{thm: induction method bounds for identical holes} dealing with the lower bounds for the smallest exact eigenvalue for convex domains with multiple holes. We also recall a few versions of McGowan's lemma as part of the proofs and present a generalisation of it to higher degree forms, the proof of which is somewhat technical and is the subject of the \hyperlink{appendix}{Appendix}.

\section{Preliminaries}\label{sec: prelim}
Let $(M^n,g)$ be a compact, connected, orientable, smooth Riemannian $n$\nobreakdash-manifold with boundary. As mentioned in the introduction, the Hodge Laplacian on the space of differential $p$\nobreakdash-forms $\Om^p(M)$, for $p\in\{0,\dots,n\}$, is defined by
\begin{align*}
\lap\p&:\Om^p(M)\to \Om^p(M)\\
\om&\mapsto (d \delta +\delta d) \om,
\end{align*}
where $d$ and $\delta$ denote the exterior derivative and the codifferential, respectively, with the convention that $d\equiv 0$ on $n$\nobreakdash-forms and $\delta\equiv 0$ on functions.  We use the standard notation $\langle.\,.\rangle$ to denote the canonical pointwise inner product of $p$\nobreakdash-forms induced by the Riemannian metric, and $\inn{L^2\Om^p(M)}{.}{.}$ to denote the $L^2$\nobreakdash-innerproduct of square integrable forms on $M$; $\vert\alpha\vert$ and $\Vert \alpha \Vert$ to denote the norms given by $\langle\alpha,\alpha\rangle^\half$ and $\inn{L^2\Om^p(M)}{\alpha}{\alpha}^\half$, respectively. The notation $\iota^*$ and $i_\normal$ stand for the pullback of a differential form with respect to the inclusion map $\iota:\pa M\hookrightarrow M$, and the interior product with respect to the unit outward normal vector field $\normal$ on the boundary.
\parnoi
\textbf{Absolute eigenvalue problem: }For $\om\in\Om^p(M)$,
\eqn{
\begin{cases}
\lap\p\om=\mu\,\om &\text{in $M$},\\
i_\normal \om=0 \text{ and } i_\normal d\om=0 &\text{on $\pa M$},
\end{cases}
}
where $\mu$ is the spectral parameter. The Hodge Laplacian with absolute boundary conditions, which generalises the Neumann problem for functions, has a discrete spectrum
\eqn{
0=\abseig_{0,1}=\dots=\abseig_{0,b_p}<\abseig_1\leq\abseig_2\leq \dots \nearrow +\infty
}
with finite multiplicities. Here, the dimension of the kernel, $b_p\in \{0\} \cup \NN$, is the absolute Betti number of $M$ (defined to be the rank of the $p^\text{th}$ homology group of $M$), a purely topological quantity independent of the metric. The $k^\text{th}$ positive eigenvalue has the variational characterisation
\eqn{
\abseig_k(M)= \min_{V_{b_p+k}}\,\, \max_{\om\in V_{b_p+k}\setminus\{0\}} \frac{\Vert d\om\Vert_{L^2\Om^{p+1}(M)}^2+\Vert \delta\om\Vert_{L^2\Om^{p-1}(M)}^2}{\Vert \om\Vert_{L^2\Om^p(M)}^2},
}
where $V_{b_p+k}$ denotes a $(b_p+k)$\nobreakdash-dimensional subspace of $p$\nobreakdash-forms with finite $H^1$\nobreakdash-norm that are tangential to the boundary (i.e., $i_\normal \om=0$). We refer to the text~\cite{hodge_text} for the definition of Sobolev norms on forms.
\par
We recall a version of the \emph{Hodge decomposition} for $p$\nobreakdash-forms~\cite[Proposition~9.8]{taylor-pde1} that is of relevance to the absolute eigenvalue problem. Given $\om\in \Om^p(M)$, there exist an $\alpha \in \Om^p(M)$ and a harmonic form $\gamma \in \Om^p(M)$ satisfying the absolute boundary conditions, $i_\normal \alpha=0$, $i_\normal d\alpha=0$, $i_\normal \gamma=0$ and  $i_\normal d\gamma=0$, such that
\eqnum{\label{hodge decomp from taylor pde}
\om=d\delta\alpha+\delta d \alpha+\gamma, 
}
and the three terms on the RHS are orthogonal in $L^2\Om^p(M)$. Note that $\gamma$, being a harmonic form satisfying the absolute boundary conditions, is closed and coclosed. As is well-known, it follows from this decomposition that the spectrum of the Hodge Laplacian splits into two families, the exact $p$\nobreakdash-spectrum
\eqn{
0<\exeig{p}{1}\leq \exeig{p}{2}\leq \dots \nearrow+\infty,
}
and the coexact $p$\nobreakdash-spectrum
\eqn{
0<\coeig{p}{1}\leq \coeig{p}{2}\leq \dots \nearrow+\infty,
}
in addition to the eigenvalue zero with finite (possibly zero) multiplicity. Moreover, the coexact $p$\nobreakdash-spectrum is identical to the exact $(p+1)$\nobreakdash-spectrum, for \mbox{$p\in\{0,\dots,n-1\}$}:
\eqn{
\coeig{p}{k}=\exeig{p+1}{k},\qquad\forall\,k\in \NN.
}
\parnoi
These spectra satisfy the variational characterisations given by
\eqn{
\exeig{p}{k}(M)= \min_{V_k^{\ex}} \max_{\om\in V_k^{\ex}\setminus\{0\}} \frac{\Vert \delta\om\Vert_{L^2\Om^{p-1}(M)}^2}{\Vert \om\Vert_{L^2\Om^p(M)}^2} \quad \text{and} \quad \coeig{p}{k}(M)= \min_{V_k^{\co}} \max_{\om\in V_k^{\co}\setminus\{0\}} \frac{\Vert d\om\Vert_{L^2\Om^{p+1}(M)}^2}{\Vert \om\Vert_{L^2\Om^p(M)}^2},
}
where $V_k^{\ex}$ (resp. $V_k^{\co}$) is a $k$\nobreakdash-dimensional subspace of the exact (resp. coexact) $p$\nobreakdash-forms with finite $H^1$\nobreakdash-norm that are tangential to the boundary. The above characterisations also hold with $V_k^{\ex}$ (resp. $V_k^{\co}$) replaced by $V_{b_p+k}^{\cl}$ (resp. $V_{b_p+k}^{\cc}$), a $(b_p+k)$\nobreakdash-dimensional subspace of the closed (resp. coclosed) $p$\nobreakdash-forms with finite $H^1$\nobreakdash-norm that are tangential to the boundary.
\parnoi
Furthermore, the following variant of the variational characterisation for the exact $p$\nobreakdash-spectrum (see, for example,~\cite[Proposition~3.1]{dod82}) would be of use in our proofs, as it has no restrictions on the boundary conditions for the test forms. For $p\in\{1,\dots,n\}$,
\eqnum{
\exeig{p}{k}(M)=\inf_{V_k} \sup_{\eta\in V_k\setminus \{0\}}\, \sup_{\theta\,:\, \eta=d\theta} \frac{\Vert \eta \Vert^2_{L^2\Om^p(M)}}{\Vert \theta \Vert^2_{L^2\Om^{p-1}(M)}},
}
where $V_k$ ranges over all $k$\nobreakdash-dimensional subspaces of the space of $L^2$\nobreakdash-integrable exact $p$\nobreakdash-forms on $M$.
\par
It is a natural choice to use the absolute boundary conditions in the study of the Hodge Laplacian spectrum for manifolds with boundary, but the following dual version of this problem, known as the relative eigenvalue problem that generalises the Dirichlet problem for functions, is also of significance: For $\om\in\Om^p(M)$,
\eqn{
\begin{cases}
\lap\p\om=\lm \om &\text{in $M$},\\
\iota^* \om=0 \text{ and } \iota^*\delta\om=0 &\text{on $\pa M$},
\end{cases}
}
where $\lm$ is the spectral parameter. It has a discrete spectrum
\eqn{
0=\releig_{0,1}=\dots=\releig_{0,b_p^\rel}<\releig_1\leq \releig_2\leq \dots\nearrow +\infty
}
with finite multiplicities. The dimension of the kernel, $b_p^\rel\in \{0\}\cup \NN$, is the relative Betti number of $(M,\pa M)$. The absolute and the relative eigenvalue problems are dual to each other and are connected by the Hodge star operator. Their spectra obey, for $p\in\{0,\dots,n\}$,
\eqn{
\abseig_{b_p+k}=\lm^{(n-p)}_{b_p^\rel+k},\qquad \forall \, k\in \NN.
}
The relative eigenvalue problem has analogous variational characterisation, and a spectral decomposition into the exact, coexact and harmonic forms with appropriate boundary conditions that can be deduced from a dual version of the Hodge decomposition stated above. Naturally, the lower bounds stated in this article for the absolute eigenvalues have counterparts for the relative eigenvalues.

\section{Lower bounds for the smallest positive eigenvalue of convex domains with a hole} \label{sec: main lower bds}

In this section, we prove Theorem~\ref{thm: first abs eigval lower bds for forms on annular} and Theorem~\ref{thm: neumann lower bd for func in terms of union}. Before that, we give an overview of the proof of Theorem~\ref{thm: first abs eigval lower bds for forms on annular}. We consider the two regimes---when the radius of the hole $R_h$ is large, or small, compared to the contact radius $R_c$. In the former case, we begin with an eigenform $\om$ of the first exact eigenvalue $\exeigpo$ and would like to find a primitive of it which satisfies a Poincar\'e-type inequality with a geometrically explicit constant, (the square of) whose inverse would serve as a lower bound for $\exeigpo$. Such an inequality is known for convex domains from the work of Chanillo-Treves~\cite[Lemma~2.1]{CT97}, and the geometric dependence of the constant there is suitable to bring out the role of $R_c$ by adapting it to our setting. We consider an open cover of our annular domain by dividing it into sufficiently small pieces that are preimages of tiny patches on the inner spherical boundary, under the orthogonal projection map of $\RR^n$ onto it. The boundary of each piece has a convex part, coming from the outer boundary of the initial domain, and a non-convex part that is a remanent of the inner boundary, in addition to the flat part of the boundary connecting these two. Now, we convexify each piece near the non-convex part using a quasi-isometry (see Figure~\ref{fig:quasi}), with explicit control on the distortion required to do so. We then apply \cite[Lemma~2.1]{CT97} to these convex domains that approximate our pieces to obtain local Poincar\'e-type inequalities for the pieces themselves, involving the restrictions of the eigenform $\omega$. However, the primitives thus obtained on these pieces do not necessarily agree on the intersections for them to be glued together. In order to obtain a global primitive, we use an appropriate partition of unity to first construct an approximate global $p$\nobreakdash-form $\nu$ from the primitives on the pieces. We then consider the difference $\om-d\nu$, and recursively update $\nu$ in $p$ steps to eventually obtain a more manageable expression for $\om-d\nu$ which involves the first derivatives of the partition of unity and certain constant functions on larger versions of the pieces (see Lemma~\ref{lem: tech lem for omega-d_eta} for a precise description). These constants, when viewed as a \v{C}ech cochain for a suitable open cover of the domain, are in fact a \v{C}ech cocycle. Furthermore, the exactness of $\omega$ implies that this system of constants is in fact a \v{C}ech coboundary. Using an explicit isomorphism between the \v{C}ech cocyles and the closed forms representing the de~Rham cohomology~\cite[Proposition~9.8]{bott-tu}, this allows us to further simplify the difference $\omega-d\nu$ and finally construct a primitive $\eta$ of $\omega$. Throughout this process, we keep track of the estimates and obtain a global Poincar\'e-type inequality as intended. In the latter regime of small holes, we proceed as follows. A spectral convergence result of Ann\'e--Colbois~\cite[Theorem~1.1]{AC93} for perforated domains with shrinking spherical holes gives explicit lower bounds when the domains are annuli whose inner and outer radii are chosen appropriately. When we have a generic domain with convex outer boundary and a spherical hole, we consider a $2$\nobreakdash-cover comprising an inner annulus of suitable dimensions and an outer annular region whose hole now has sufficiently large radius for the lower bounds from the earlier regime to be applicable. Finally, we apply to this $2$\nobreakdash-cover a gluing lemma of Cheeger--McGowan~\cite[Lemma~2.3]{mcg93} that, under certain topological assumptions which hold in our setting, gives lower bounds for the exact $p$\nobreakdash-spectrum of a domain in terms of the lower bounds on the exact $p$\nobreakdash-spectrum of the two pieces and the exact $(p-1)$\nobreakdash-spectrum of their intersection. This gluing procedure gives lower bounds on $\exeigpo$ for $p\in\{2,\dots,n-1\}$. For the case $p=1$ (which also corresponds to the Neumann eigenvalue for functions), the role of this gluing lemma is replaced by our Theorem~\ref{thm: neumann lower bd for func in terms of union}. We now proceed to the proof of Theorem~\ref{thm: first abs eigval lower bds for forms on annular}.

Let $\dom=\Dom\setminus \overline{\BnRh} $ be an annular domain in $\RR^n$, where $\Dom$ is a convex domain in $\RR^n$ and \mbox{$\BnRh\subsetneq\Dom$} is a Euclidean $n$\nobreakdash-ball of radius $\Rh$. Let $R_c$ and $D$ be the contact radius and the diameter of $\dom$, and $\SRh=\pa\BnRh$ denote the round sphere of radius $\Rh$. Fix $n\geq 2$ and $p\in \{1,\dots,n-1\}$. We divide the proof into two cases-- when $\Rh\geq \frac{R_c}{2}$, and when $\Rh< \frac{R_c}{2}$.
\parnoi
\textbf{Case 1: }$\Rh\geq \frac{R_c}{2}$.
\parnoi
Let $\om\in\Om^p(\dom)$ be an eigenform associated with $\exeigpo(\dom)$. Using the Hodge decomposition~\eqref{hodge decomp from taylor pde}, let $\alpha\df \delta \hat{\alpha}$ be the coexact $(p-1)$\nobreakdash-form such that $d\alpha=\om$, where $\hat{\alpha}$ satisfies the absolute boundary conditions $i_\normal \hat{\alpha}=0$ and $i_\normal d \hat{\alpha}=0$. So, we have from $\lap d\alpha=\exeigpo(\dom)d\alpha$ that $d(\lap \alpha-\exeigpo(\dom)\alpha)=0$. Further, $\lap \alpha-\exeigpo(\dom)\alpha=\delta(\lap \hat{\alpha}-\exeigpo(\dom)\hat{\alpha})$, with $i_\normal (\lap \hat{\alpha}-\exeigpo(\dom)\hat{\alpha})=0$ since $i_\normal \hat{\alpha}=0$, $i_\normal d \delta \hat{\alpha}=i_\normal \om=0$, and $i_\normal \delta d \hat{\alpha}=\delta i_\normal d \hat{\alpha}=0$. It follows from the following lemma that $\lap \alpha =\exeigpo(\dom) \alpha$, i.e., $\alpha$ is a coexact eigenform with the eigenvalue $\exeigpo(\dom)$.
\begin{lem}
For a compact Riemannian manifold  $(M,g)$ with boundary, if $\tilde{\om}\in \Om^p(M)$ is closed and coexact (say, $\tilde{\om}=\delta \hat{\om}$) with $i_\normal \hat{\om}=0$, then $\tilde{\om}=0$.
\end{lem}
\vspace{-1.2em}
\begin{proof}
From the Hodge decomposition~\eqref{hodge decomp from taylor pde}, we may write
\eqn{
\tilde{\om}=d\tilde{\beta}+\delta \tilde{\alpha} +\tilde{h},
}
where $\tilde{\alpha}$ is an exact form with $i_\normal \tilde{\alpha}=0$ and $\tilde h$ is a harmonic form satisfying the absolute boundary conditions (hence closed and coclosed). As $\tilde\om$ is closed, we have that $d\delta \tilde \alpha=0$, which, using Green's identity, implies
\eqn{
\inn{L^2\Om^p(M)}{\delta \tilde \alpha}{\delta \tilde \alpha}=\inn{L^2\Om^p(M)}{d \delta \tilde \alpha}{\tilde\alpha}-\inn{L^2\Om^p(\pa M)}{\iota^*\delta \tilde\alpha}{i_\normal d\tilde \alpha}=0,
}
hence the Hodge decomposition of $\tilde\om$ has no $\delta \tilde \alpha$ component. Now,
\begin{align*}
\inn{L^2\Om^p(M)}{\tilde\om}{\tilde\om}&=\inn{L^2\Om^p(M)}{\delta \hat \om}{d \tilde \beta + \tilde h}\\
&=\inn{L^2\Om^p(M)}{\delta \hat \om}{d \tilde \beta}+\inn{L^2\Om^p(M)}{\delta \hat \om}{\tilde h}\\
&=\inn{L^2\Om^p(M)}{\delta \delta \hat \om}{\tilde \beta}+\inn{L^2\Om^p(\pa M)}{\iota^*\tilde\beta}{i_\normal \delta\hat \om}+\inn{L^2\Om^p(M)}{\hat \om}{d \tilde h}-\inn{L^2\Om^p(\pa M)}{\iota^*\tilde h}{i_\normal \hat \om}\\
&=0.
\end{align*}
\end{proof}
\vspace{-1.5em}
\noi Hence,
\eqn{
\exeigpo(\dom)=\frac{\Vert d\alpha \Vert^2_{L^2\Om^p(\dom)}}{\Vert \alpha \Vert^2_{L^2\Om^{p-1}(\dom)}}.
}
It is difficult to directly find explicit geometric lower bounds for the term on the RHS, but we observe that $\alpha$ minimises the $L^2$\nobreakdash-norm among all $\eta \in \Om^{p-1}(\dom)$ such that $d \eta=\om$: If $\eta$ is any such form, then $\eta-\alpha$ is closed, so from the Hodge decomposition~\eqref{hodge decomp from taylor pde}, there exists a harmonic form $h$ satisfying the absolute boundary conditions $i_\normal h=0$ and $i_\normal dh=0$, and an exact form $\beta$ with appropriate boundary conditions, such that $\eta-\alpha=\beta+h$. Further, $\alpha$, $\beta$ and $h$ would be mutually orthogonal, and we have
\eqn{
\Vert\alpha\Vert_{L^2\Om^{p-1}(\dom))}\leq \Vert \eta\Vert_{L^2\Om^{p-1}(\dom))}.
}
This implies that
\eqnum{\label{flexible rayleight quotient lower bd}
\exeigpo(\dom)\geq \frac{\Vert \om \Vert^2_{L^2\Om^p(\dom))}}{\Vert \eta \Vert^2_{L^2\Om^{p-1}(\dom))}},
}
giving us the flexibility in terms of the choice of $\eta$. We adapt the methods from \cite{CT97} to construct a primitive $\eta\in \Om^{p-1}(\dom)$ (i.e., $d\eta=\om$) such that its $L^2$\nobreakdash-norm is controlled by that of $\om$ explicitly in terms of the geometry of $\dom$. More precisely, we prove the following:
\begin{prop}\label{prop: main upper bd on rayleigh}
There exists a primitive $\eta\in\Om^{p-1}(\dom)$ of $\om$ such that
\eqn{
\frac{\Vert \eta \Vert_{L^2\Om^{p-1}(\dom)}}{\omnorm} \leq K R_c \left(\frac{D}{R_c}\right)^{(n+1)p+\frac{n}{2}}\left(\RhRc\right)^{(n-1)(\frac{3p}{2}+2)+p}.
}
\end{prop}
\noi
We would then conclude from~\eqref{flexible rayleight quotient lower bd} that, when $\Rh\geq \frac{R_c}{2}$
\eqn{
\exeigpo(\dom)\geq \frac{\Vert \om \Vert^2_{L^2\Om^
p(\dom)}}{\Vert \eta \Vert^2_{L^2\Om^{p-1}(\dom)}} \geq K \frac{1}{R_c^2}\left(\frac{R_c}{D}\right)^{2(n+1)p+n}\left(\frac{R_c}{\Rh}\right)^{(3n-1)p+4(n-1)}.
}
\parnoi
Let us proceed with the proof of Proposition~\ref{prop: main upper bd on rayleigh}. We utilise the following estimate for convex domains, by applying it to sufficiently small pieces of the annular domain that are convexified with a control on the volumes and the diameters. We then use these local estimates to deduce a global Poincar\'e-type inequality, using an explicit isomorphism between \v{C}ech and de~Rham cohomology.
\begin{lem}{\cite[Lemma~2.1]{CT97}}
For a bounded, convex domain $\mathcal{D}\in\RR^n$, there exists a constant $K(n)>0$ such that, given any closed form $\tilde\om\in \Om^p(\overline{\mathcal{D}})$, for $p\in\{1,\dots,n\}$, there exists a solution $\tilde\phi\in\Om^{p-1}(\overline{\mathcal{D}})$ to the equation $d\tilde\phi=\tilde\om$ that satisfies
\eqn{
\Vert\tilde\phi\Vert_{L^2\Om^{p-1}(\mathcal{D})}\leq K \frac{\diam(\mathcal{D})^{n+1}}{\vol(\mathcal{D})}\Vert\tilde\om \Vert_{L^2\Om^p(\mathcal{D})}.
}
\end{lem}
\noi 
The proof of this lemma from~\cite{CT97} also holds for domains which are piecewise smooth.
\begin{figure}[H]
    \centering
    \includegraphics[width=0.6\linewidth]{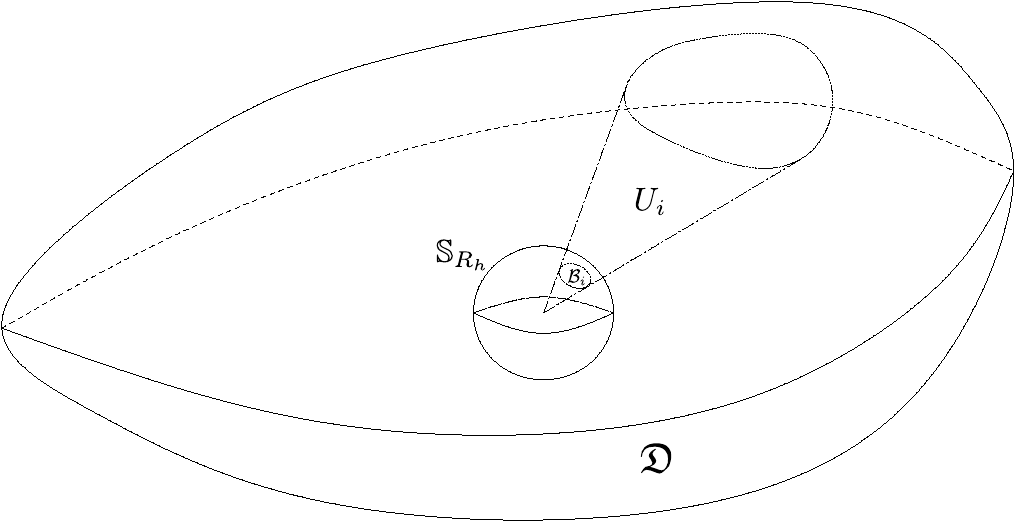}
    \caption{}
\end{figure}
\parnoi
For convenience, we choose the centre of $\SRh$ to be the origin. Define
\eqn{
r_0\df KR_c,
}
for sufficiently small $K$ described later by~\eqref{r_0}. Choose an open cover of $\SRh$ given by balls $\{\calB_i\}_{i=1}^N$ of radii $\frac{r_0}{4^n}$ on $\SRh$, with their centers $\{t_i\}_{i=1}^N$ at least $\frac{r_0}{4^n}$ apart mutually with respect to the round metric. Observe that
\eqnum{\label{bound on N}
N\leq \frac{\vol(\SRh)}{\vol(\half\calB_i)} \leq K\left(\Rhrc\right)^{n-1}\leq K \left(\RhRc\right)^{n-1},
}
where $\half \calB_i$ denotes the ball on $\SRh$ concentric to $\calB_i$ with half its radius. Now, $\{U_i\df \dom\cap\proj^{-1}(\calB_i)\}_{i=1}^N$ will be an open cover for $\dom$, where $\proj$ is the orthogonal projection from $\RR^n$ onto $\SRh$. Our choice of $r_0$ above also ensures that $\{\calB_i\}$ are geodesically convex, and so are all the nontrivial intersections $\cap_{k=1}^{k_0} \calB_{i_k}$, for $k_0\in \{1,\dots,N\}$. It then follows that $\{U_i\}_{i=1}^N$ is a \emph{good cover} of $\dom$, i.e., all the nontrivial intersections $\cap_{k=1}^{k_0} U_{i_k}$ are contractible for all $k_0\in \{1,\dots,N\}$. We use this property later while invoking the \v{C}ech cohomology. Furthermore, we have, for $i\in\{1,\dots,N\}$,
\eqnum{\label{vol lower bound}
KD^n\geq \vol(U_i) \geq KR_c r_0^{n-1}\geq K R_c^n.
}
For $q\in \{0,\dots,N-1\}$ and a multi-index $I=(i_0,\dots,i_q)$ such that $1\leq i_0<\dots<i_q\leq N$, denote by $U_I$ the intersection~$\cap_{k=0}^q U_{i_k}$. Let $S_q$ be the set of such multi-indices of length $q+1$ with~$U_I\neq \varnothing$. Note that
\eqn{
\card(S_q)\leq N(N-1)\dots (N-q)\leq N^{q+1}.
}
Define, for $\lm\in\{0,\dots,n\}$,
\eqn{
U_i^\lm\df\dom\cap\proj^{-1}\left(\calB \Big(t_i,\frac{4^\lm}{4^n}r_0\Big)\right),
}
and
\eqn{
U_I^\lm\df U_{i_0}^\lm \cap \dots \cap U_{i_q}^\lm,
}
for each $I=(i_0,\dots,i_q)\in S_q$. We have that, for $\lm\in\{0,\dots,n-1\}$,
\eqn{
U_{i_k}^\lm\subset U_I^{\lm+1},\qquad \forall \, k\in\{0,\dots,q\}, \forall \, I \in S_q.
}
Let $\{\tilde\varphi_i\}_{i=1}^N$ be a partition of unity subordinate to the cover $\{\calB_i\}_{i=1}^N$ of $\SRh$ such that $\Vert d\tilde\varphi_i\Vert_\infty\leq \frac{K}{r_0}$ for all $i\in\{1,\dots,N\}$, for some constant $K(n)$. This gives rise to a partition of unity $\{\varphi_i\df \tilde\varphi_i\circ \proj\}_{i=1}^N$ subordinate to the cover $\{U_i\}_{i=1}^N$ of $\dom$, also satisfying the bound $\Vert d\varphi_i\Vert_\infty\leq \frac{K}{R_c}$ for all $i\in\{1,\dots,N\}$.
\begin{lem}
Let $\lm\in\{1,\dots,n\}$. There exists a constant $K(n)$ such that, given an exact form $\tilde\om\in\Om^p\left(\overline{U_I^\lm}\right)$, there exists $\tilde\phi\in\Om^p\left(\overline{U_I^\lm}\right)$ such that $d\tilde\phi=\tilde\om$ on $U_I^\lm$ and
\eqnum{\label{control for local solution}
\Vert \tilde\phi \Vert_{L^2\Om^{p-1}(U_I^\lm)} \leq K \frac{D^{n+1}}{R_c^n}\Vert \tilde\om \Vert_{L^2\Om^p(U_I^\lm)}. 
}
\end{lem}
\begin{proof}
For $t\in \SRh$, set
\eqn{
\calO_\theta\df \dom\cap\proj^{-1}\left(\calB(t,\Rh\theta)\right),\qquad \theta\in(0,\frac{\pi}{2}).
}
\begin{figure}[H]
    \centering
    \includegraphics[width=0.5\linewidth]{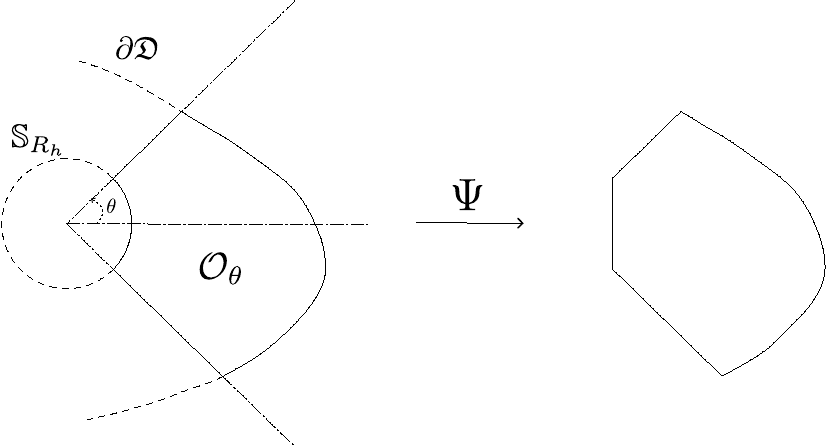}
    \caption{}
    \label{fig:quasi}
\end{figure}
\noi
We construct a quasi-isometry as follows, that transforms a sufficiently small piece of $\dom$ into a convex domain with control on the Jacobian of this transformation. Choose a coordinate system whose origin is at the centre of $\BnRh$ and such that the point $t$ corresponds to $(\Rh,0,\dots,0)$. Define
\begin{align*}
\Phi : \calO_\theta&\to \RR^n\\
(x_1,\dots,x_n)&\mapsto\frac{R\cos \theta}{x_1}
(x_1,\dots,x_n),
\end{align*}
where $R=\sqrt{\sum_{i=1}^n x_i^2}$. Let $\alpha:[\Rh,\infty)\to [0,1]$ be a smooth cutoff function that is equal to $1$ on $[\Rh,\Rh+\frac{R_c}{3}]$ and vanishing on $[\Rh+\frac{2R_c}{3},\infty)$, whose gradient is controlled by $|\alpha'|\leq \frac{K}{R_c}$. The required quasi-isometry is then given by
\eqn{
\Psi(x_1,\dots,x_n)\df \alpha(R) \Phi(x_1,\dots,x_n)+(1-\alpha(R))(x_1,\dots,x_n).
}
We have, for $k,l\in\{1,\dots,n\}$,
\eqn{
\frac{\pa \Psi_k}{\pa x_l}= \alpha(R)\frac{\pa \Phi_k}{\pa x_l}(x_1,\dots,x_n)+\alpha'(R)\frac{x_l}{R}\Phi_k(x_1,\dots,x_n)+(1-\alpha(R))\delta_{kl}-\alpha'(R)\frac{x_lx_k}{R}.
}
Substituting the expression for $\Phi$ and using that
\begin{align*}
x_1&=R\cos(\theta(x))\geq R\cos \theta,\\
|x_i|&\leq R \sin(\theta(x)) \leq R\sin \theta, \qquad \forall \, i\in \{2,\dots,n\},\\
\alpha&=\alpha'\equiv 0 \qquad \text{on }\left[\Rh+\frac{2R_c}{3},\infty\right),
\end{align*}
one can deduce that
\eqn{
\Vert \jac\Psi-\id_n \Vert_{\max} \leq K\left(1+\RhRc\right)\theta\leq K_0 \RhRc \theta.
}
By the equivalence of the max norm and the operator norm for matrices, let $K_1(n)\geq 1$ be a constant such that
\eqn{
K_1^{-1}\Vert.\Vert_{\max}\leq \Vert.\Vert_{\op}\leq  K_1\Vert.\Vert_{\max}.
}
Set
\eqn{
\theta_0\df \frac{1}{3K_1K_0\RhRc},
}
requiring without loss of generality that the RHS is less than $\frac{\pi}{2}$.
Then, for $\theta\in(0,\theta_0]$,
\eqn{
\Vert \id-\jac \Psi\Vert_{\op}<1,
}
hence the series $\sum_{k=0}^\infty(\id-\jac\Psi)^k$ converges to $(\id-(\id-\jac \Psi))^{-1}=(\jac \Psi)^{-1}=\jac\Psi^{-1}$. That is,
\eqn{
\jac\Psi^{-1}=\id+(\id-\jac \Psi)+(\id-\jac \Psi)^2+\dots,
}
and
\begin{align*}
\Vert \jac \Psi\inv-\id\Vert_{\op}&=\Vert (\id-\jac \Psi)+(\id-\jac \Psi)^2+\dots\Vert_{\op}\\
&\leq \Vert \id-\jac \Psi \Vert_{\op}+\Vert \id-\jac \Psi \Vert_{\op}^2+\dots\\
& \leq K_1K_0\RhRc\theta+\left(K_1K_0\RhRc\theta\right)^2+\dots\\
&=\frac{K_1K_0\RhRc\theta}{1-K_1K_0\RhRc\theta}\\
&\overset{\text{using }\theta\leq\theta_0}{\leq}\half.
\end{align*}
Thus,
\eqnum{\label{jacobian control}
\max\left\{\Vert \jac \Psi-\id\Vert_{\op}, \Vert \jac \Psi\inv-\id\Vert_{\op}\right\}\leq\half,\qquad \forall \,\theta\in(0,\theta_0].
}
Set
\eqnum{\label{r_0}
r_0\df\Rh\theta_0=K R_c,
}
so that $r_0<\frac{\pi}{2}\Rh$ and the balls $\calB(t_i,r_0)$ are geodesically convex on $\SRh$.
\parnoi
Let $(U_I^\lm)'\subset \RR^n$ be the image of $U_I^\lm\subset \calO_{\theta_0}$ under the diffeomorphism $\Psi$. We have
\begin{align*}
\vol(U_I^\lm)&=\int_{U_I^\lm}\dv\\
&=\int_{(U_I^\lm)'}|\det(\jac\Psi\inv)|\,\dv\\
&\leq\Vert\jac \Psi\inv\Vert_{\op}^n \int_{(U_I^\lm)'}\dv\\
&\overset{\text{using~}\eqref{jacobian control}}{\leq} \left(\frac{3}{2}\right)^n \vol((U_I^\lm)').
\end{align*}
We also have that $\diam((U_I^\lm)')=\diam(U_I^\lm)$, since $(U_I^\lm)'$ is in fact the convex-hull of $U_I^\lm$.
\parnoi
We recall~\cite[Lemma~2.1]{CT97}: For a bounded, convex domain $\mathcal{D}\in\RR^n$, there exists a constant $K(n)>0$ such that, given any closed form $\tilde\om\in \Om^p(\overline{\mathcal{D}})$, for $p\in\{1,\dots,n\}$, there exists a solution $\tilde\phi\in\Om^{p-1}(\overline{\mathcal{D}})$ to the equation $d\tilde\phi=\tilde\om$ that satisfies
\eqn{
\Vert\tilde\phi\Vert_{L^2\Om^{p-1}(\mathcal{D})}\leq K \frac{\diam(\mathcal{D})^{n+1}}{\vol(\mathcal{D})}\Vert\tilde\om \Vert_{L^2\Om^p(\mathcal{D})}.
}
Applying this to the exact form $(\Psi\inv)^*\tilde\om$, the pullback of $\tilde\om$ with respect to $\Psi\inv:(U_I^\lm)'\to U_I^\lm$, there exists $\hat \phi\in \Om^{p-1}((U_I^\lm)')$ with $d\hat\phi=(\Psi\inv)^*\tilde\om$ such that
\eqn{
\left(\int_{(U_I^\lm)'} |\hat\phi|^2\,\dv\right)^\half\leq K \frac{\diam((U_I^\lm)')^{n+1}}{\vol((U_I^\lm)')}\left(\int_{(U_I^\lm)'} |(\Psi\inv)^*\tilde\om|^2\,\dv\right)^\half.
}
Choose $\tilde\phi=\Psi^*\hat\phi$. Using diameter and volume comparison from before, in addition to the control on the Jacobian, we get
\begin{align*}
\left(\int_{U_I^\lm} |\tilde \phi|^2\,\dv\right)^\half &= \left(\int_{(U_I^\lm)'} |\det\jac \Psi\inv|\,|(\Psi\inv)^*\tilde \phi|^2\,\dv\right)^\half\\
&\leq \left(\frac{3}{2}\right)^{\frac{n}{2}} \left(\int_{(U_I^\lm)'} |\hat\phi|^2\,\dv\right)^\half\\
&\leq K \frac{\diam((U_I^\lm)')^{n+1}}{\vol((U_I^\lm)')}\left(\int_{(U_I^\lm)'} |(\Psi\inv)^*\tilde\om|^2\,\dv\right)^\half\\
&\leq K \frac{\diam(U_I^\lm)^{n+1}}{\vol(U_I^\lm)}\left(\int_{U_I^\lm} |\det \jac\Psi|\, |\tilde\om|^2\,\dv\right)^\half\\
&\leq K \frac{\diam(\dom)^{n+1}}{R_c^n} \left(\int_{U_I^\lm} |\tilde\om|^2\,\dv\right)^\half,
\end{align*}
where we use that, for $\lm\in\{1,\dots,n\}$,
\eqn{
\vol(U_I^\lm)\geq \vol(U_j))\overset{\text{by }\eqref{vol lower bound}}{\geq} K R_c^n.
}Hence the lemma.
\end{proof}
\noi
Let $\vareps(i,I\setminus i)$ denote the sign of the permutation that orders $(i,I\setminus i)$.
\begin{lem}\label{lem: tech lem for omega-d_eta}
There exists $\nu\in\Om^{p-1}(\dom)$, and, 
for every $I\in S_p$, a constant function $\eta_I$ on $U_I^1$ such that
\eqnum{\label{cech-derham isom expression}
\om-d\nu=-\sum_{I\in S_p}\left(\sum_{i\in I}\vareps(i,I\setminus i)\varphi_i\,d\varphi_{I\setminus i}\right)\eta_I,
}
with
\begin{gather}
\Vert \nu\Vert_{L^2\Om^{p-1}(\dom)}\leq K R_c \left(\frac{D}{R_c}\right)^{(n+1)p}\left(\RhRc\right)^{(n-1)p}\omnorm,\\
\left(\vol(U_I^1)\right)^\half|\eta_I|=\Vert \eta_I \Vert_{L^2(\dom)}\leq K \left(\frac{D^{n+1}}{R_c^n}\right)^p \omnorm, \qquad \text{and} \label{ineq with vol term}\\
\sum_{j\in J}\vareps(j,J\setminus j)\, \eta_{J\setminus j}=0, \qquad \forall\, J\in S_{p+1}.\label{cech cocycle condition}
\end{gather}
Here, $d\varphi_K$ denotes the $(k+1)$\nobreakdash-form $d\varphi_{i_0}\wedge\dots\wedge d\varphi_{i_k}$ for a multi-index $K=(i_0,\dots,i_k)$. 
\end{lem}
\begin{proof}
Adapting the construction from \cite[Section~2]{CT97}, we prove the following property by induction: For any $q\in \{1,\dots,p\}$, there exists a closed form $\eta_I\q \in\Om^{p-q}(U_I^{p-q+1})$ for each $I\in S_q$, and $\nu\q \in \Om^{p-1}(\dom)$ such that
\begin{gather}
\om-d\nu\q=-\sum_{I\in S_q}\left(\sum_{i\in I}\vareps(i,I\setminus i)\varphi_i d\varphi_{I\setminus i}\right)\wedge \eta_I\q,\label{indn prop1}\\
\Vert \nu\q\Vert_{L^2\Om^{p-1}(\dom)}\leq K R_c \left(\frac{D}{R_c}\right)^{(n+1)q}\left(\RhRc\right)^{(n-1)q}\omnorm,\label{indn prop2}\\
\Vert \eta_I\q \Vert_{L^2\Om^{p-q}(U^{p-q+1})}\leq K \left(\frac{D^{n+1}}{R_c^n}\right)^q \omnorm,\label{indn prop3}
\intertext{and}
\sum_{j\in J}\vareps(j,J\setminus j) \,\eta\q_{J\setminus j}=0, \qquad \forall \,J\in S_{q+1}.\label{indn prop4}
\end{gather}
For the base case $q=1$, we proceed as follows. We have from~\eqref{control for local solution} that there exists, for each $i\in \{1,\dots,N\}$, an \mbox{$\eta^{(1)}_i \in \Om^{p-1}(U_i^p)$} such that $d\eta^{(1)}_i=\om$ and
\eqn{
\Vert \eta^{(1)}_i \Vert_{L^2\Om^{p-1}(U_i^p)}\leq K \frac{D^{n+1}}{R_c^n}\omnorm.
}
We then set $\nu^{(1)}\df\sum_{i=1}^N \varphi_i\eta^{(1)}_i$ so that
\begin{align*}
\Vert \nu^{(1)} \Vert_{L^2\Om^{p-1}(\dom)}&\leq K N \frac{D^{n+1}}{R_c^n}\omnorm\\
&\overset{\text{by }\eqref{bound on N}}{\leq} K R_c \left(\frac{D}{R_c}\right)^{n+1}\left(\RhRc\right)^{n-1}\omnorm.
\end{align*}
Using that $\sum_{i=1}^N d\varphi_i=0$ on $\dom$, a simple computation from \cite[Section~2]{CT97} gives
\eqn{
\om-d\nu^{(1)}=-\sum_{(i,j)\in S_1} (\varphi_id\varphi_j-\varphi_jd\varphi_i)\wedge \eta^{(2)}_{ij},
}
where $\eta^{(2)}_{ij}\in \Om^{p-1}(U_{ij}^p)$ is the closed form defined to be $\eta^{(1)}_i-\eta^{(1)}_j$ for $i<j$, satisfying
\begin{align*}
\left(\int_{U^p_{ij}}\left|\eta^{(2)}_{ij}\right|^2\,\dv\right)^\half&\leq \left(\int_{U^p_{ij}}\left|\eta^{(1)}_i\right|^2+\left|\eta^{(1)}_j\right|^2+2\langle \eta^{(1)}_i,\eta^{(1)}_j\rangle\,\dv\right)^\half\\
&\leq \left(\int_{U^p_{ij}}\left|\eta^{(1)}_i\right|^2+\left|\eta^{(1)}_j\right|^2+2\max\left\{\left|\eta^{(1)}_i\right|^2,\left|\eta^{(1)}_j\right|^2\right\}\,\dv\right)^\half\\
&\leq K \frac{D^{n+1}}{R_c^n}\omnorm.
\end{align*}
Furthermore, for $(i,j,k)\in S_2$, we have
\eqn{
u_{jk}-u_{ik}+u_{ij}=u_j-u_k-u_i+u_k+u_i-u_j=0.
}
Suppose that the induction hypothesis holds for some $q\in\{1,\dots,p-1\}$. We now show that this would imply that the property holds for $q+1$. Again, we have from~\eqref{control for local solution} that for each $I\in S_q$, there exists $\nu\q_I\in \Om^{p-q-1}(U^{p-q+1}_I)$ such that $d\nu\q_I=\eta\q_I$ on $U^{p-q+1}_I$ and
\eqn{
\Vert \nu\q_I \Vert_{L^2\Om^{p-q-1}(U^{p-q+1}_I)}\leq K \frac{D^{n+1}}{R_c^n}\Vert \eta\q_I \Vert_{L^2\Om^{p-q}(U^{p-q+1}_I)}\leq K \left(\frac{D^{n+1}}{R_c^n}\right)^{q+1}\omnorm,
}
where we use the induction hypothesis in the second inequality. Now, define the $(p-1)$-form on $\dom$,
\eqn{
\nu\qq\df\nu\q+\sum_{I\in S_q}\sum_{i\in I} \vareps(i,I\setminus i)\varphi_i\,d\varphi_{I\setminus i}\wedge \nu\q_I
}
and for each $J\in S_{q+1}$, define the $(p-q-1)$\nobreakdash-form on $U^{p-q}_j$,
\eqn{
\eta\qq_J=-(q+1)\sum_{j\in J}\vareps(j, J\setminus j)\nu\q_{J\setminus j}.
}
The forms $\{\eta\qq_J\}$ are closed by property~\eqref{indn prop4} for $q$, and their $L^2$\nobreakdash-norms are bounded by
\begin{align*}
\Vert \eta\qq_J\Vert_{L^2\Om^{p-q-1}(U_j^{p-q})}&\leq (q+1)\sum_{j\in J} \Vert \nu\q_{J\setminus j}\Vert_{L^2\Om^{p-q-1}(U_j^{p-q})}\\
&\leq (q+1)(q+2)K\left(\frac{D^{n+1}}{R_c^n}\right)^{q+1}\omnorm\\
&\leq K\left(\frac{D^{n+1}}{R_c^n}\right)^{q+1}\omnorm.
\end{align*}
One can then repeat the computations from \mbox{\cite[Section~2]{CT97}} to show that
\begin{gather*}
\om-d\nu\qq=-\sum_{J\in S_{q+1}}\left(\sum_{j\in J}\vareps(j,J\setminus j)\varphi_j d\varphi_{J\setminus j}\right)\wedge \eta_J\qq
\intertext{and}
\sum_{j\in J}\vareps(j,J\setminus j) \,\eta\qq_{J\setminus j}=0, \qquad \forall \,J\in S_{q+2}.
\end{gather*}
Lastly, since $\supp(\varphi_i\,d\varphi_{I\setminus i})\subset U^{p-q+1}_I$, we get from the definition of $\nu\qq$,
\begin{align*}
\Vert \nu\qq \Vert_{L^2\Om^{p-1}(\dom)}-\Vert \nu\q \Vert_{L^2\Om^{p-1}(\dom)}&\leq \sum_{I\in S_q}\sum_{i\in I}\Vert \varphi_i\,d\varphi_{I\setminus i}\Vert_\infty \Vert \nu\q_I \Vert_{L^2\Om^{p-q-1}(U^{p-q+1}_I)}\\
&\leq K \card(S_q) (q+1) \Inv{R_c^q} \left(\frac{D^{n+1}}{R_c^n}\right)^{q+1}\omnorm\\
&\leq K N^{q+1} \Inv{R_c^q} \left(\frac{D^{n+1}}{R_c^n}\right)^{q+1}\omnorm\\
&\overset{\text{by }\eqref{bound on N}}{\leq} K \left(\RhRc\right)^{(n-1)(q+1)}\Inv{R_c^q} \left(\frac{D^{n+1}}{R_c^n}\right)^{q+1}\omnorm,
\end{align*}
which implies that
\begin{align*}
\Vert \nu\qq \Vert_{L^2\Om^{p-1}(\dom)}&\leq\Vert \nu\q \Vert_{L^2\Om^{p-1}(\dom)}+K R_c \left(\frac{D}{R_c}\right)^{(n+1)(q+1)}\left(\RhRc\right)^{(n-1)(q+1)}\omnorm\\
&\leq K R_c \left(\frac{D}{R_c}\right)^{(n+1)q}\left(\RhRc\right)^{(n-1)q}\left[1+\left(\frac{D}{R_c}\right)^{n+1}\left(\RhRc\right)^{n-1}\right]\omnorm\\
&\leq K R_c \left(\frac{D}{R_c}\right)^{(n+1)(q+1)}\left(\RhRc\right)^{(n-1)(q+1)}\omnorm,
\end{align*}
using that $\frac{D}{R_c}\geq 1$ and $\RhRc\geq \half$.
Thus, the lemma follows by induction, where we finally choose $\nu=\nu\p \in \Om^{p-1}(\dom)$ and the constant functions $\eta_I=\eta\p_I$ on $U_I^1$.
\end{proof}
\noi
Recall that
\eqn{
\vol(U_I^1)\geq \vol(U_i) \geq K R_c^n.
}
So, by~\eqref{ineq with vol term},
\eqnum{\label{bound on u_I}
|\eta_I|\leq K \frac{D^{(n+1)p}}{R_c^{n(p+\half)}}\, \omnorm.
}
Consider the family of constants $(\eta_I)_{I\in S_p}$ as a \v{C}ech cochain with respect to the good cover $\{U_i\}_{i=1}^N$. In view of~\eqref{cech cocycle condition}, it is a \v{C}ech cocycle. It then follows from~\eqref{cech-derham isom expression} that the \v{C}ech cohomology class of this cocycle is same as the de~Rham cohomology class of $\om$ (see \cite[Propostion~9.5]{bott-tu}). As $\om$ is exact, it further means that $(\eta_I)_{I\in S_p}$ is in fact a coboundary. Using this and proceeding as in \cite[Section~2]{CT97}, we get that
\eqn{
d\eta=\om,
}
where $\eta$ is constructed as
\eqn{
\eta=\nu+\sum_{H\in S_{p-1}}\sum_{i\in H}\vareps(i,I\setminus i)c_H\left(1+\frac{1}{p}-\frac{p+1}{2}\varphi_i-\sum_{j\notin H}\varphi_j \right)\varphi_i\,d\varphi_{H\setminus i}.
}
Here, $c_H$ are constants which satisfy the bounds
\begin{align*}
\sum_{H\in S_{p-1}} |c_H|&\leq K N^{\frac{p}{2}+1}\sum_{I\in S_p}|\eta_I|\\
&\overset{\text{by \eqref{bound on u_I}}}{\leq} K N^{\frac{p}{2}+1} \card(S_p) \frac{D^{(n+1)p}}{R_c^{n(p+\half)}}\, \omnorm\\
&\leq K \left(\RhRc\right)^{(n-1)(\frac{3p}{2}+2)} \frac{D^{(n+1)p}}{R_c^{n(p+\half)}}\, \omnorm.
\end{align*}
Thus,
\begin{align*}
\frac{\Vert \eta \Vert_{L^2\Om^{p-1}(\dom)}}{\omnorm}&\leq K \left(R_c \left(\frac{D}{R_c}\right)^{(n+1)p}\left(\RhRc\right)^{(n-1)p}\right)\\&\hspace{10em}+K\left(\left(\RhRc\right)^{(n-1)(\frac{3p}{2}+2)} \frac{D^{(n+1)p}}{R_c^{n(p+\half)}}\left(\RhRc\right)^p\frac{1}{R_c^{p-1}}D^\frac{n}{2}\right)\\
&\leq K R_c \left(\frac{D}{R_c}\right)^{(n+1)p+\frac{n}{2}}\left(\RhRc\right)^{(n-1)(\frac{3p}{2}+2)+p},
\end{align*}
which completes the proof of Proposition~\ref{prop: main upper bd on rayleigh}.
{
\parnoi
\textbf{Case 2: }$\Rh< \frac{R_c}{2}$.\nopagebreak
\parnoi
We need the following property which follows from a straightforward adaptation of a result of Ann\'e and Colbois \cite[Theorem~1.1]{AC93} concerning the spectral convergence for perforated domains, to the framework of manifolds with boundary.
}
\begin{prop}\label{prop: anne colbois spectral convergence adaptation}
Let $M$ be a compact Riemannian $n$\nobreakdash-manifold with boundary ($n\geq 2$) and $M_\eps$ be the manifold obtained by removing a finite number of non\nobreakdash-overlapping geodesic balls of radius at~most $\eps>0$, for $\eps$ smaller than the injectivity radius of $M$. Then,
\eqn{
\lim_{\eps\to 0}\exeigpo(M_\eps)=\exeigpo(M),\qquad \forall\,p\in\{0,\dots,n-1\}.
}
\end{prop}
\vspace{-1em}
\noi
If $\dom$ is an annulus whose outer radius is $R_c$ and inner radius at most $\frac{R_c}{2}$, consider the annulus $\dom'$ of outer radius $1$ and inner radius at most $\half$, obtained from scaling $\dom$ by a factor of $\Inv{R_c}$. By Proposition~\ref{prop: anne colbois spectral convergence adaptation} and continuity of eigenvalues with respect to the metric, we have
\eqn{
\exeigpo(\dom')\geq K,
}
which implies that
\eqn{
\exeigpo(\dom)\geq \frac{K}{R_c^2}.
}
\parnoi
When $\dom$ is a generic domain with $\Rh<\frac{R_c}{2}$, we proceed as follows.
\parnoi
For $n\geq 3$ and $p\in\{2,\dots,n-1\}$, we use a version of \cite[Lemma 1]{GP95} which extends \cite[Lemma~2.3]{mcg93} by Cheeger and McGowan (where it was stated for closed manifolds but the proof is essentially the same for absolute boundary conditions):
\begin{lem}\label{lem: modified McGowan}
Let $(M^n,g)$ be a compact Riemannian manifold with boundary $(n\geq 3)$ such that $M=U_1\cup U_2$ for some smooth subdomains $U_1,U_2$ with $H^{p-1}(U_1\cap U_2, \RR)=0$. Let $\{\rho_1,\rho_2\}$ be a partition of unity subordinate to $\{U_1,U_2\}$ with
\eqn{
c_\rho\df \max\{\Vert \grad\rho_1 \Vert_{\infty}^2, \Vert \grad\rho_2 \Vert_{\infty}^2\}.
}
Then we have, for $p\in\{2,\dots,n-1\}$,
\eqn{
\exeigpo(M)\geq \frac{K}{\left( \frac{1}{\exeigpo(U_1)}+\frac{1}{\exeigpo(U_2)}\right)\left(\frac{c_\rho}{\exeig{p-1}{1}(U_1\cap U_2)}+1\right)}.
}
\end{lem}
\noi
We write $\dom=\dom_1\cup \dom_2$, where $\dom_1$ is the annular region whose outer boundary is the same as that of $\dom$ and the inner boundary is a sphere of radius $\frac{R_c}{2}$ (concentric to $\SRh$), and $\dom_2$ is the annular region of inner radius $\Rh$ and outer radius $R_c$ (also concentric to $R_c$). Choose a partition of unity $\{\hat\rho_1,\hat\rho_2\}$ subordinate to $\{\dom_1,\dom_2\}$ such that
\eqn{
c_\rho\df \max\{\Vert \grad\hat\rho_1 \Vert_{\infty}^2, \Vert \grad\hat\rho_2 \Vert_{\infty}^2\}\leq \frac{K}{(R_c-\Rh)^2} \leq \frac{K}{(\frac{R_c}{2})^2}.
}
Applying Lemma~\ref{lem: modified McGowan} and noting that $H^{p-1}(\dom_1\cap\dom_2, \RR)=H^{p-1}(\Snm, \RR)=0$ for $n\geq 3$ and $p\in \{2,\dots,n-1\}$, we get,
\begin{align*}
\exeigpo(\dom)&\geq \frac{K}{\left(R_c^2 \left(\frac{D}{R_c}\right)^{2(n+1)p+n}\left(\frac{\Rh}{R_c}\right)^{(3n-1)p+4(n-1)}+R_c^2\right)\left(\frac{R_c^{-2}}{R_c^{-2}}+1\right)}\\
&\geq K \frac{1}{R_c^2}\left(\frac{R_c}{D}\right)^{2(n+1)p+n}.
\end{align*}
For $n\geq 3$ and $p\in\{2,\dots,n-1\}$, we thus have, for any annular region,
\begin{align}
\exeigpo(\dom) &\geq K \frac{1}{R_c^2}\left(\frac{R_c}{D}\right)^{2(n+1)p+n}\min\left\{1,\left(\frac{R_c}{\Rh}\right)^{(3n-1)p+4(n-1)}\right\} \notag\\
&\geq K \frac{1}{R_c^2}\left(\frac{R_c}{D}\right)^{n(2n-1)}\min\left\{1,\left(\frac{R_c}{\Rh}\right)^{3(n-1)(n+1)}\right\},\label{longer exprn of main lower bd p>=2}
\end{align}
giving the required lower bound, using that $\Rhh \le D$.
\parnoi
When $n\geq 2$ and $p=1$, for $\Rh<\frac{R_c}{2}$, we take $\dom_1$ and $\dom_2$ as before and obtain from Theorem~\ref{thm: neumann lower bd for func in terms of union},
\begin{align*}
\exeig{1}{1}(\dom)=\mu_1(\dom)&\geq \frac{1}{32} \frac{\vol(\dom_1\cap\dom_2)}{\vol(M)} \min\left\{\mu_1(\dom_1),\mu_1(\dom_2)\right\}\\
&\geq K \left(\frac{R_c}{D}\right)^n \min\left\{\frac{1}{R_c^2}\left(\frac{R_c}{D}\right)^{3n+2}\left(\frac{R_c}{\Rh}\right)^{7n-5},\frac{1}{R_c^2}\right\}\\
&\geq K \frac{1}{R_c^2}\left(\frac{R_c}{D}\right)^{n}\min\left\{1,\left(\frac{R_c}{D}\right)^{3n+2}\left(\frac{R_c}{\Rh}\right)^{7n-5}\right\}.
\end{align*}
Therefore, for $n\geq 2$, we have for any annular region,
\begin{align}
\exeig{1}{1}&\geq K \frac{1}{R_c^2}\min\left\{\left(\frac{R_c}{D}\right)^{3n+2}\left(\frac{R_c}{\Rh}\right)^{7n-5},\left(\frac{R_c}{D}\right)^{n}\min\left\{1,\left(\frac{R_c}{D}\right)^{3n+2}\left(\frac{R_c}{\Rh}\right)^{7n-5}\right\}\right\}\notag\\
&\geq K \frac{1}{R_c^2}\left(\frac{R_c}{D}\right)^{n}\min\left\{1,\left(\frac{R_c}{D}\right)^{3n+2}\left(\frac{R_c}{\Rh}\right)^{7n-5}\right\},\label{longer exprn of main thm p=1}
\end{align}
thus concluding the proof of Theorem~\ref{thm: first abs eigval lower bds for forms on annular}.\qed
\parnoi
\textbf{Proof of Theorem~\ref{thm: neumann lower bd for func in terms of union}: }Let $f$ be an eigenfunction associated with ${\mu_1}(M)$, normalised by
\eqn{
\int_M f^2=1,
}
and $\bar f_U$, $\bar f_V$ be the averages of $f$ over $U$ and $V$, respectively, i.e.,
\eqn{
\bar f_U=\frac{\int_U f }{\vol(U)}\qquad \text{and}\qquad \bar f_V
=\frac{\int_V f }{\vol(V)}.
}
Then, $f-\bar f_U$ and $f-\bar f_V$ are $L^2$\nobreakdash-orthogonal to the constants, and the variational characterisation for the first positive Neumann eigenvalue gives,
\eqnum{\label{ineq 1 in proof of lemma for lower bd for func}
{\mu_1}(U)\int_{U\cap V}(f-\bar f_U)^2  \leq {\mu_1} (U)\int_U (f-\bar f_U)^2  \leq \int_U |df|^2\leq \int_M |df|^2  ={\mu_1}(M),
}
and
\eqnum{\label{ineq 2 in proof of lemma for lower bd for func}
{\mu_1}(V)\int_{U\cap V}(f-\bar f_V)^2\leq {\mu_1}(M).
}
By writing
\eqn{
|\bar f_U-\bar f_V|=\frac{1}{\vol(U\cap V)^\half}\left(\int_{U \cap V} \left((f-\bar f_U)-(f-\bar f_V)\right)^2 \right)^\half,
}
we obtain from inequalities \eqref{ineq 1 in proof of lemma for lower bd for func} and \eqref{ineq 2 in proof of lemma for lower bd for func} that
\eqnum{\label{ineq 3 in proof of lemma for lower bd for func}
|\bar f_U-\bar f_V|\leq \left(\frac{{\mu_1}(M)}{\vol(U\cap V)}\right)^\half \left(\frac{1}{\sqrt{{\mu_1}(U)}}+\frac{1}{\sqrt{{\mu_1}(V)}}\right).
}
From \eqref{ineq 2 in proof of lemma for lower bd for func},
\eqn{
\left\vert \left(\int_V f^2 \right)^\half -\vol(V)^\half |\bar f_V| \right\vert\leq \left(\int_V(f-\bar f_V)^2\right)^\half\leq \sqrt{\frac{{\mu_1}(M)}{{\mu_1}(V)}}.
}
Thus,
\eqn{
\vol(V)^\half |\bar f_V|\geq \left(\int_V f^2 \right)^\half- \sqrt{\frac{{\mu_1}(M)}{{\mu_1}(V)}}.
}
Without loss of generality, we may assume that
\eqn{
\int_V f^2 \geq \half.
}
Then,
\eqn{
|\bar f_V| \geq \frac{1}{\vol(V)^\half}\left(\frac{1}{\sqrt{2}}-\sqrt{\frac{{\mu_1}(M)}{{\mu_1}(V)}}\right).
}
If $\frac{1}{\sqrt{2}}-\sqrt{\frac{{\mu_1}(M)}{{\mu_1}(V)}}\leq 0$, we obtain
\eqn{
{\mu_1}(M)\geq \half {\mu_1}(V),
}
which implies the required lower bound of the theorem. Otherwise, since $\vol(V)\leq \vol(M)$,
\eqnum{\label{ineq 4 in proof of lemma for lower bd for func}
|\bar f_V|\geq \Inv{\vol(M)^\half}\left( \frac{1}{\sqrt{2}}-\sqrt{\frac{{\mu_1}(M)}{{\mu_1}(V)}}\right).
}
On the other hand, $f$ is orthogonal to constants, so
\eqn{
0=\int_M f =\int_U f +\int_{V\setminus U} f =\vol(U)(\bar f_U-\bar f_V)+\vol(M) \bar f_V+\int_{V\setminus U} f -\vol(V\setminus U) \bar f_V,
}
i.e.,
\eqn{
0=\vol(U)(\bar f_U-\bar f_V)+\vol(M) \bar f_V+\int_{V\setminus U} (f-\bar f_V).
}
Therefore,
\eqn{
|\bar f_V|\leq \frac{1}{\vol(M)}\left(\vol(U)|\bar f_U-\bar f_V|+\int_{V\setminus U}|f-\bar f_V|  \right).
}
Using the Cauchy-Schwarz inequality and \eqref{ineq 2 in proof of lemma for lower bd for func}, we get
\eqn{
|\bar f_V|\leq \frac{1}{\vol(M)}\left(\vol(U)|\bar f_U-\bar f_V|+ \vol(V\setminus U)^\half \sqrt{\frac{{\mu_1}(M)}{{\mu_1}(V)}}\right),
}
which gives
\eqnum{\label{ineq 5 in proof of lemma for lower bd for func}
|\bar f_V|\leq |\bar f_U-\bar f_V|+\Inv{\sqrt{\vol(M)}}\sqrt{\frac{{\mu_1}(M)}{{\mu_1}(V)}}.
}
Finally, we have from \eqref{ineq 3 in proof of lemma for lower bd for func}, \eqref{ineq 4 in proof of lemma for lower bd for func} and \eqref{ineq 5 in proof of lemma for lower bd for func} that
\begin{multline*}
\Inv{\sqrt{\vol(M)}}\left(\frac{1}{\sqrt{2}}-\sqrt{\frac{{\mu_1}~(M)}{{\mu_1}(V)}}\right)\\\leq \frac{\sqrt{{\mu_1}(M)}}{\sqrt{\vol(U\cap V)}} \left(\frac{1}{\sqrt{{\mu_1}(U)}}+\frac{1}{\sqrt{{\mu_1}(V)}}\right)+\Inv{\sqrt{\vol(M)}}\sqrt{\frac{{\mu_1}(M)}{{\mu_1}(V)}},
\end{multline*}
which implies that
\begin{align*}
\frac{1}{\sqrt{2}}&\leq 2\left(1+\left(\frac{\vol(M)}{\vol(U \cap V)}\right)^\half\right)\left(\frac{{\mu_1}(M)}{\min\{{\mu_1}(U),{\mu_1}(V)\}}\right)^\half\\
&\leq 4\left(\frac{\vol(M)}{\vol(U \cap V)}\right)^\half \left(\frac{{\mu_1}(M)}{\min\{{\mu_1}(U),{\mu_1}(V)\}}\right)^\half.
\end{align*}
The theorem follows. \qed
\section{Construction of small eigenvalues on annular regions} \label{sec: small eigvals}
In this section, we prove Theorem~\ref{thm: domains with small eigval}. In view of the correspondence between the exact part of the $p$\nobreakdash-spectrum and the coexact part of the $(p-1)$\nobreakdash-spectrum, it will suffice to construct, for each $p\in\{0,\dots,n-2\}$, a sequence of domains $(\ann_\eps^p)_{\eps\in(0,\half)}$ such that $\lim_{\eps\to 0}\coeigpo(\ann_\eps^p)=0$. Note that $\coeig{n-1}{1}$ would be bounded below since
\eqn{
\coeig{n-1}{1}=\exeig{n}{1}=\mu^{(n)}_1=\lm_1,
}
and the first Dirichlet eigenvalue $\lm_1$ is uniformly bounded below for domains with bounded diameter, by the Faber-Krahn inequality. The main idea is to create an extra dimension of cohomology on the limiting domain as the contact radius tends to zero. We then use the corresponding harmonic form to construct on the domains appropriate test forms, the supports of whose exterior derivatives have their volumes converging to zero at a required rate, and estimate the associated Rayleigh quotients that appear in the variational characterisation for the first coexact eigenvalue.
\begin{figure}[H]
    \centering
    \includegraphics[width=0.45\linewidth]{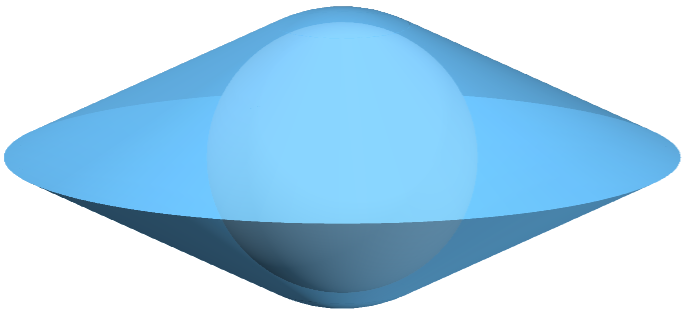}\qquad\qquad
    \includegraphics[width=0.36\linewidth]{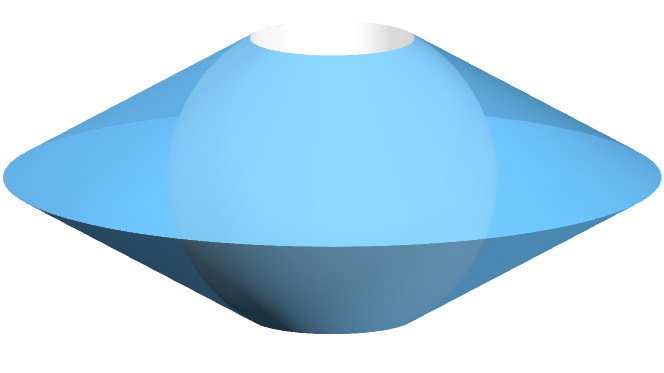}\qquad\qquad
    \includegraphics[width=0.8\linewidth]{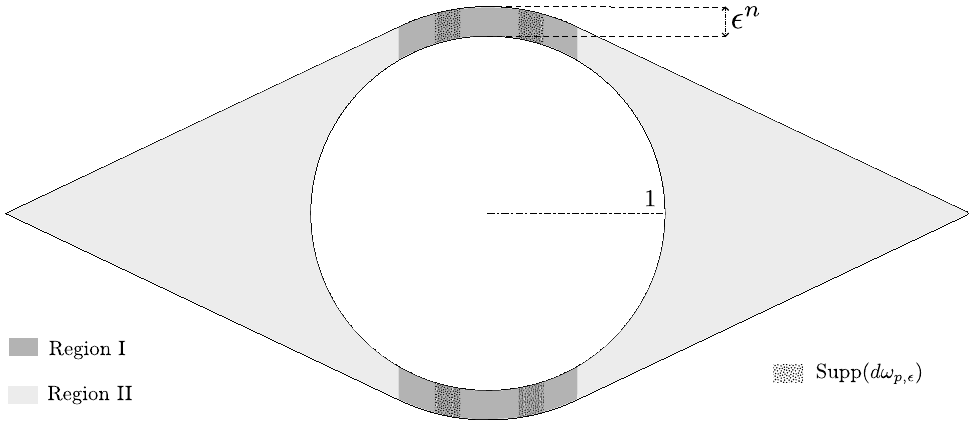}
    \caption{$\ann^p_\eps$, $\ann^p_0$, and cross-section of $\ann^p_\eps$ for $n=3$, $p=1$}
    \label{fig:construction}
\end{figure}
\parnoi
Writing $\RR^n=\RR^{n-p-1}\times \RR^{p+1}$, we use the cylindrical coordinates given by the combination of polar coordinates $(\rho,\thetatil_1,\dots,\thetatil_{n-p-2})$ on $\RR^{n-p-1}$ with the volume element $\rho^{n-p-2}\,d\rho\,d\thetatil_1\dots d\thetatil_{n-p-2}$ and $(r, \theta_1,\dots, \theta_{p})$ on $\RR^{p+1}$ with the volume element $r^p\,dr\,d\theta_1,\dots d\theta_p$. For $\eps\in[0,\half)$, define $\ann^p_\eps\subset\RR^n$ to be the interior of the region given by
\eqn{
\underbrace{\left\{r\leq \half;\, 1\leq \sqrt{\rho^2+r^2}\leq 1+\eps^n\right\}}_{\text{Region I}}\bigcup \underbrace{\left\{r\geq \half;\,1\leq\sqrt{\rho^2+r^2}\text{ and }\rho\sqrt{(1+\eps^n)^2-\frac{1}{4}}+\frac{r}{2}\leq (1+\eps^n)^2\right\}}_{\text{Region II}}
}
after smoothening it around $\left\{(\rho,r)=\left(0,2\left(\half+\eps^n\right)^2\right)\right\}$, which can be done without affecting it beyond an arbitrarily small neighbourhood and such that the convexity of the domain and the tangency of the boundary to $\theta_i$\nobreakdash-directions are preserved. In other words, $\ann^p_\eps$ is obtained by considering a spherical shell of inner radius $1$ and thickness $\eps^n$, and adding to it the region enclosed by the tangential extensions, at $\{r=\half\}$, of the parts of its outer boundary. The boundary of $\ann^p_\eps$ has $C^1$\nobreakdash-regularity at the joint involving the tangential extension, while the other joint where the tangential extensions from the two sides meet and close the boundary can be smoothened, as mentioned above. See Figure~\ref{fig:construction}.
\parnoi
Clearly, the diameter of $\ann^p_\eps$ is bounded independent of~$\eps$. Note that $\{\ann^p_\eps\}_{\eps>0}$ are homotopic to $\Snm$ and have the cohomology
\eqn{
H^p(\ann^p_\eps)=\begin{cases} 0 &\text{if }p\neq 0,\\ \RR &\text{if }p=0,\end{cases}
}
whereas the limiting domain $\ann^p_0$ is homotopic to $\SSp$ and hence
\eqn{
H^p(\ann^p_0)=\begin{cases} \RR &\text{if }p\neq 0,\\ \RR^2 &\text{if }p=0.\end{cases}
}
If $p\neq 0$, we may write $\RR^n=\RR^{n-p-1}\times \RR^{p+1}$, equipping $\RR^{n-p-1}$ with the Cartesian coordinates $(x_1,\dots,x_{n-p-1})$ and $\RR^{p+1}$ with the spherical coordinates $(r, \theta_1,\dots,\theta_p)$, where $\theta_j\in [0,\pi]$ for \mbox{$j\in \{1,\dots,p-1\}$} and $\theta_p\in [0,2\pi)$. The metric in these coordinates takes the form
\eqn{
ds^2=\bigoplus_{i=1}^{n-p-1}dx_i^2\oplus dr^2\oplus r^2\left( \bigoplus_{j=1}^p \alpha_j^2\,d\theta_j^2\right),
}
where $\{\alpha_j\}$ are smooth functions of $\{\theta_j\}$ given by
\eqn{
\alpha_j(\theta_1,\dots,\theta_p)\df \prod_{k=1}^{j-1} \sin\theta_k,\qquad j\in \{1,\dots,p\}.
}
Then, for $p\neq0$, the $p$\nobreakdash-form
\eqn{
\om_p\df \alpha_1\dots\alpha_p\, d\theta_1\wedge\dots\wedge d\theta_p
}
defined on $\RR^{n-p-1}\times (\RR^{p+1}\setminus\{0\})$ is a harmonic representative of its $p^\text{th}$ cohomology:
\eqn{
d\om_p=d(\alpha_1\dots\alpha_p)\wedge d\theta_1\wedge\dots\wedge d\theta_p=0
}
and
\eqn{
\delta \om_p=\pm \star d \star \om_p=\pm\star d(r^{-p} dr\wedge dx_1\wedge \dots\wedge dx_{n-p-1})=0.
}
The restriction of $\om_p$ to $\ann^p_0$ is a non-zero harmonic form satisfying the absolute boundary conditions, since the $\{\theta_i\}$ directions are parallel to the boundary.
\parnoi
If $p=0$, we consider the function $\om_0$ that equals $1$ on one of the connected components and $-1$ on the other.
\parnoi
Let $\chi_\eps$ be a smooth, non\nobreakdash-negative, increasing function defined on $\RR_{\geq 0}$ such that it is equal to $0$ on $[0,\frac{\eps}{3}]$ and $1$ on $[\frac{2\eps}{3}, \infty)$. We construct the family of test forms
\eqn{
\om_{p,\eps}\df \chi_\eps(r) \om_p
}
defined on $\ann^p_\eps$, for $\eps>0$. Note that the $p$\nobreakdash-form $\om_{p,\eps}$ is tangential to the boundary of $\ann^p_\eps$ (i.e., $i_\normal \om_{p,\eps}=0$ on~$\pa \ann^p_\eps$), and when $p=0$, it is orthogonal to the constant functions on $\ann^0_\eps$. Furthermore, it is also coclosed:
\eqn{
\delta \om_{p,\eps}=\delta (\chi_\eps\om_p)=\chi_\eps\, \delta \om_p-i_{\grad\chi_\eps}\om_p=0,
}
because $\om_p$ is coclosed and $\grad \chi_\eps$ is orthogonal to the angular directions corresponding to $\{\theta_j\}$.
We intend to show that
\eqn{
\lim_{\eps\to 0} \frac{\int_{\ann^p_\eps} |d\om_{p,\eps}|^2\,\dv}{\int_{\ann^p_\eps} |\om_{p,\eps}|^2\,\dv}=0.
}
Since $\chi_\eps=1$ on $[\half,\infty)$ for every $\eps \in (0,\half)$, we have
\eqn{
\int_{\ann^p_\eps} |\om_{p,\eps}|^2\,\dv\geq
\int_{\ann^p_0} |\om_p|^2\,\dv>0,
}
i.e., the denominator in the above limit is bounded away from $0$, independent of $\eps$. Hence, it would suffice to show that 
\eqn{
\lim_{\eps\to 0} \int_{\ann^p_\eps} |d\om_{p,\eps}|^2\,\dv =0.
}
We proceed as follows.
\eqn{
d\om_{p,\eps}=d(\chi_\eps\om_p)=d(\chi_\eps)\wedge \om_p \pm \chi_\eps\wedge d\om_p=\chi_\eps'(r)\,dr\wedge \om_p.
}
So,
\begin{align*}
|d\om_{p,\eps}|^2\,\dv&=(\chi'_\eps)^2(dr\wedge \om_p)\wedge \star (dr\wedge \om_p)\\
&=\pm(\chi'_\eps)^2 (dr\wedge \om_p)\wedge (r^{-p}dx_1\wedge \dots\wedge dx_{n-p-1})\\
&=(\chi'_\eps)^2 r^{-2p} \dv.
\end{align*}
Choosing the function $\chi$ defined earlier such that $\Vert\chi'\Vert_{\infty}\leq K \eps\inv$, we get
\begin{align*}
\int_{\ann^p_\eps} |d\om_{p,\eps}|^2\,\dv &\leq K\eps^{-2p-2} \vol\left(\supp \ann^p_\eps\right)\\
&= K\eps^{-2p-2} \vol\left(\ann^p_\eps\cap \left\{\frac{\eps}{3}\leq r \leq \frac{2\eps}{3}\right\}\right).
\end{align*}
The volume term on the RHS is
\begin{align*}
\vol\left(\ann^p_\eps \cap \left\{\frac{\eps}{3} \leq r \leq \frac{2\eps}{3}\right\}\right)&= K(n,p) \int_{r=\frac{\eps}{3}}^\frac{2\eps}{3} \int_{\rho=\sqrt{1-r^2}}^{\sqrt{(1+\eps^n)^2-r^2}} r^p \rho^{n-p-2}\, d\rho\,dr\\
&=K \int_{\frac{\eps}{3}}^\frac{2\eps}{3} r^p \left[\left((1+\eps^n)^2-r^2\right)^{\frac{n-p-1}{2}}-\left(1-r^2\right)^{\frac{n-p-1}{2}}\right]\,dr,
\end{align*}
which can be estimated as follows. We have that
\eqn{
r=\Appr{\eps} \qquad \text{and} \qquad (1+\eps^n )^2 = 1+2\eps^n+\Appr{\eps^{2n}},
}
so
\eqn{
\frac{(1+\eps^n)^2-1}{1-r^2}=\Appr{\eps^n}.
}
Then
\eqn{
\left((1+\eps^n)^2-r^2\right)^{\frac{n-p-1}{2}}=(1-r^2)^{\frac{n-p-1}{2}}\left[ 1+\half(n-p-1)\frac{(1+\eps^n)^2-1}{1-r^2}+\Appr{\eps^{2n}}\right],
}
implying that
\eqn{
\vol\left(\supp \ann^p_\eps\right)=\int_{\frac{\eps}{3}}^{\frac{2\eps}{3}} r^p \,f(r) \,dr,
}
with $f(r)=\Appr{\eps^n}$. Hence,
\eqn{
\vol\left(\supp \ann^p_\eps\right)=\Appr{\eps^{n+p+1}},
}
giving that
\eqn{
\int_{\ann^p_\eps} |d\om_{p,\eps}|^2\,\dv = \Appr{\eps^{n-p-1}}.}
Since $p\leq n-2$,
\eqn{
\lim_{\eps\to 0} \int_{\ann^p_\eps} |d\om_{p,\eps}|^2\,\dv =0.
}
Thus,
\eqn{
\lim_{\eps\to 0} \frac{\int_{\ann^p_\eps} |d\om_{p,\eps}|^2\,\dv}{\int_{\ann^p_\eps} |\om_{p,\eps}|^2\,\dv}=0,
}
which from the variational characterisation
\eqn{
\coeigpo(M)=\inf_{\begin{subarray}{c}\phi\in\Om^p(M)\\\delta \phi=0\\i_\normal \phi=0 \end{subarray}}\frac{\Vert d\phi\Vert^2_{L^2\Om^{p+1}(M)}}{\Vert \phi\Vert^2_{L^2\Om^p(M)}} \qquad (\text{for $\phi$ with finite $H^1$\nobreakdash-norm}), 
}
implies that $\coeig{p}{1}(\ann^p_\eps)\to 0$ as $\eps\to 0$. \qed
\section{Eigenvalue lower bounds for convex domains with multiple holes} \label{sec: lower bds multiple holes}
In this section, we first present lower bounds for certain higher order exact eigenvalues of convex domains with multiple holes of varying sizes, with no further geometric assumptions on the domain. Next, we prove a generalised version of Theorem~\ref{thm: induction method bounds for identical holes} which gives lower bounds for the first exact eigenvalue under additional geometric hypothesis on the domain. We also construct a family of convex domains with multiples holes, satisfying the hypothesis of Theorem~\ref{thm: induction method bounds for multi holes}, having arbitrarily small eigenvalues as the partition parameter tends to zero, with the diameter of the domain and the radius of the largest hole uniformly bounded. 

\parnoi
Let $\dom\subset \RR^n$ be a perforated domain with $\holes$ many holes, that is of the form
\eqn{
\dom=\Dom\setminus \bigsqcup_{i=1}^\holes \overline{\BB_{c_i,r_i}},
}
where $\Dom$ is a convex domain in $\RR^n$ and $\BB_{c_i,r_i}$ denotes the Euclidean $n$-ball of radius $r_i>0$, centered at $c_i\in \RR^n$. Here, $\{c_i\}$ and $\{r_i\}$ are assumed to be such that the closures of the balls are mutually disjoint and strictly contained in $\Dom$.
\begin{figure}[H]
    \centering
    \includegraphics[width=0.45\linewidth]{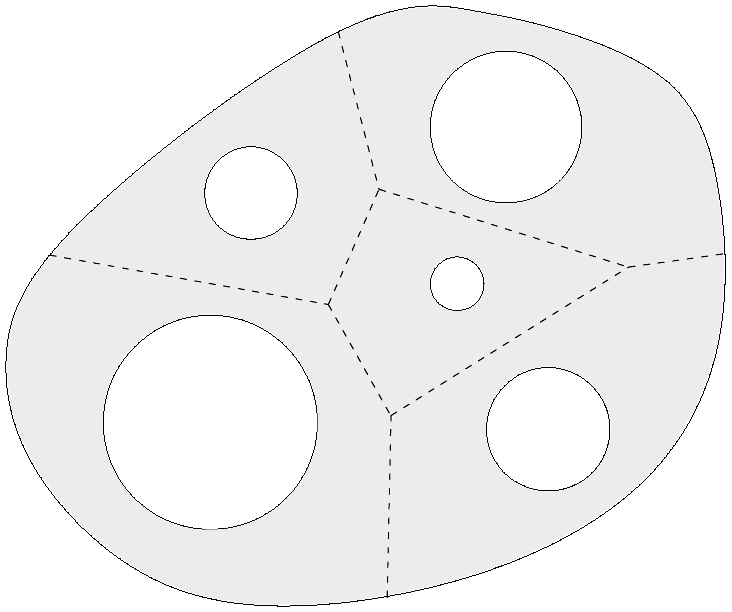}
    \hspace{3em}
    \includegraphics[width=0.45\linewidth]{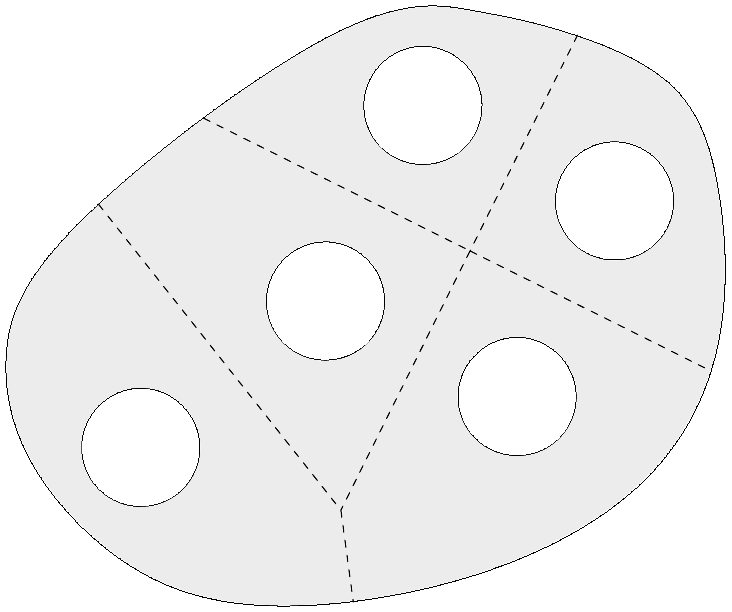}
    \caption{Laguerre-Voronoi diagram}
    \label{fig: LV diagram}
\end{figure}
\parnoi
Consider the \emph{Laguerre-Voronoi diagram} (also known as the power diagram) ~\cite[Section~6.2]{text-voronoi} of $\dom$, with centers $\{c_i\}_{i=1}^\holes$ and the corresponding weights $\{r_i^2\}_{i=1}^\holes$. This gives a partition of $\dom$ into $\holes$ many subregions (known as ``cells'') given by
\eqn{
\mathcal{C}_i\df\left\{x\in \overline{\dom}: \Vert x-c_i\Vert^2-r_i^2\leq \Vert x-c_j\Vert^2-r_j^2,\,\,\forall\,j\in\{1,\dots,\holes\} \right\},
}
such that $\{\mathcal{C}_i\cup\BB_{c_i,r_i}\}$ are convex. See~Figure~\ref{fig: LV diagram}.
\parnoi
We define the \emph{partition parameter} of $\dom$ associated with the Laguerre-Voronoi diagram to be
\eqn{
\Rpart(\dom)\df \min_{i\in \{1,\dots,\holes\}} R_c(\mathcal{C}_i),
}
where $R_c$ is the contact radius defined earlier for annular domains.
Let $k_p$ denote the number of non-empty intersections of order $p$ in the Laguerre-Voronoi partition of~$\dom$.
\begin{thm}\label{thm: mcgowan method bounds for multi holes} Let $\dom\subset \RR^n$ be a smooth, bounded, convex domain with $\holes$ many holes, as above. Then the $(1+k_p)^\text{th}$ exact eigenvalue of $\dom$ satisfies the following lower bounds.
\begin{enumerate}[(i)]
\item For $n\geq 3$ and $p\in \{2,\dots,n-1\}$,
\eqn{
\exeig{p}{1+k_p}(\dom)\geq \frac{K}{\holes^{2p-1}}\frac{1}{D^2} \left(\frac{\Rpart}{D}\right)^{5n^2+n-9}\min\left\{1,\Inv{\Rpart^{2(n-3)}}\right\}.
}
\item For $n\geq 2$ and $p=1$,
\eqn{
\exeig{1}{1+k_1}(\dom) \geq \frac{K}{\holes} \frac{1}{D^2}\left(\frac{\Rpart}{D}\right)^{11n-7}.
}
\end{enumerate}
\end{thm}
\parnoi
Although we use the Laguerre-Voronoi partition in Theorem~\ref{thm: mcgowan method bounds for multi holes} (and Theorem~\ref{thm: induction method bounds for multi holes} below) to have an ``optimal'' value of the partition parameter, especially when the holes are of identical size, the lower bounds also hold if we consider any partition of $\dom$ into $\holes$ many pieces, each of which is a convex domain with a hole in the interior, and use the partition parameter and $k_p$ corresponding to it.
\parnoi
In order to have the best possible partition parameter $\Rpart$ that we use in our estimates, it may seem more natural to consider the additively weighted Voronoi diagram (whose partition parameter would be equal to the contact radius $\Rch$) where the cells are given by
\begin{samepage}
\begin{align*}
\mathfrak{C}_i&\df\left\{x\in \overline{\dom}: \dist(x,\BB_{c_i,r_i})\leq \dist(x,\BB_{c_j,r_j}),\,\forall\,j\in\{1,\dots,\holes\} \right\}\\
&=\left\{x\in \overline{\dom}: \Vert x-c_i\Vert-r_i\leq \Vert x-c_j\Vert-r_j,\,\forall\,j\in\{1,\dots,\holes\} \right\},
\end{align*}
\end{samepage}\hspace{-.3em}
but the outer boundaries of these cells are no more convex, unless the holes are of identical radii. The convexity of the outer boundary is important for our purposes as we intend to use in our proof the lower bounds derived earlier for convex domains minus a hole. Nevertheless, the Laguerre-Voronoi diagram coincides with the standard Voronoi diagram when the holes are of identical size.
\begin{remark}
The contact radius for convex domains with identical holes defined earlier is also applicable for holes of varying size:
\eqn{
\Rch(\dom)= \min\left\{r_c,\frac{d_h}{2}\right\},
}
where $r_c$ is the distance between the outer boundary and a hole closest to it and $d_h$ is the shortest of the distances between any two holes. Further, if $\hat{d}_h$ denotes the largest of the distances between any two holes, and $\rhh$, $\Rhh$ denote the smallest and the largest of the radii of the holes, an elementary computation shows that the partition parameter satisfies
\eqn{
\left(\frac{d_h+2\hat{r}_h}{\hat{d}_h+2\hat{R}_h}\right)\Rch\leq \Rpart \leq \Rch,
}
and that $\Rpart=\Rch$ when the holes are of identical size.
\end{remark}
These lower bounds from Theorem~\ref{thm: mcgowan method bounds for multi holes} in general do not give lower bounds for the smallest exact eigenvalue $\exeigpo$, unless the Laguerre-Voronoi partition of $\dom$ (or, if exists, any other partition with the cells having a convex outer boundary) has no $p$\nobreakdash-intersections. However, we can derive the lower bounds of Theorem~\ref{thm: induction method bounds for multi holes} below for $\exeigpo$ under additional geometric hypothesis on $\dom$, using an induction-like approach.
\begin{hyp}\label{hyp: geom condition on domain with holes}
The Laguerre-Voronoi diagram $\{\mathcal{C}_i\}_{i=1}^\holes$ of the domain $\dom$, up to reordering of the cells, is such that
\eqn{
\bigcup_{i=1}^\ell \,\mathcal{C}_i \cup \, \BB_{c_i,r_i}
}
is convex for each $\ell\in\{1,\dots,\holes\}$.
\end{hyp}

\begin{thm}\label{thm: induction method bounds for multi holes}
Let $\dom\subset \RR^n$ be smooth, bounded, convex domain with $\holes$ many holes, as above. Suppose that $\dom$ satisfies the geometric condition given by Hypothesis~\ref{hyp: geom condition on domain with holes}. Then we have the following lower bounds for the first exact eigenvalue.
\begin{enumerate}[(i)]
\item For $n\geq 3$ and $p\in \{2,\dots,n-1\}$,
\eqn{
\exeigpo(\dom)\geq K \frac{1}{D^2} \left(\frac{\Rpart}{D}\right)^{5n^2-n-7+2\holes}.
}
\item For $n\geq 2$ and $p=1$,
\eqn{
\exeig{1}{1}(\dom)\geq K \frac{1}{D^2} \left(\frac{\Rpart\rhh^{n-1}}{D^n}\right)^{\holes-1}\left(\frac{\Rpart}{D}\right)^{11n-7}.
}
\end{enumerate}
\end{thm}
\noi
The exponents of the ratios $\Rpart/D$ and $\Rpart\hat{r}_h^{n-1}/D^n$ in Theorem~\ref{thm: induction method bounds for multi holes} depend on the number of holes, in contrast to the lower bounds of Theorem~\ref{thm: mcgowan method bounds for multi holes}.
\begin{remark} \label{rmk: weaker hyp for induction method multi holes}
Note that Hypothesis~\ref{hyp: geom condition on domain with holes} is always true when $\dom$ has at most two holes. Furthermore, Theorem~\ref{thm: induction method bounds for multi holes} can be adapted to give similar lower bounds for a larger class of domains than those satisfying this hypothesis, for instance, when the cells of the Laguerre-Voronoi partition may be arranged into a finite nested sequence, with the subset of the cells from each inner sequence satisfying Hypothesis~\ref{hyp: geom condition on domain with holes} and the set of unions of the cells in each inner sequence themselves satisfying an analogue of Hypothesis~\ref{hyp: geom condition on domain with holes}. This allows, for example, the domain depicted in~Figure~\ref{fig:ExampleIII}. Another simple example is the following domain with four holes, which does not directly satisfy Hypothesis~\ref{hyp: geom condition on domain with holes} but we can apply this remark to get lower bounds on $\exeigpo$.
\begin{figure}[H]
    \centering
    \includegraphics[width=0.6\linewidth]{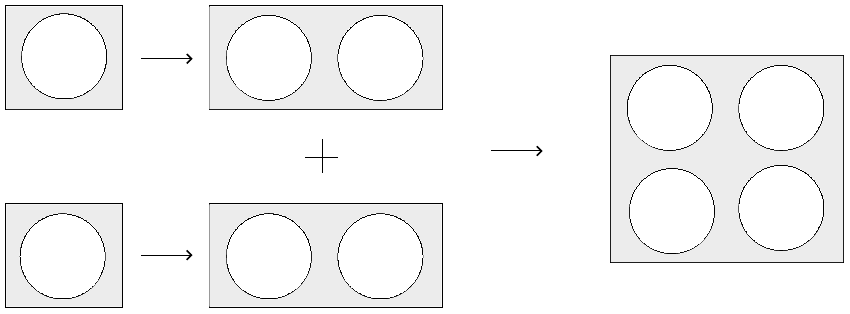}
    \caption{Example I\!V}
    \label{fig:ExampleIV}
\end{figure}
\end{remark}
\noi
The following theorem emphasises the role of the partition parameter $\Rpart$ in the lower bounds of Theorem~\ref{thm: induction method bounds for multi holes}, for $n\geq 3$ and $p\in\{2,\dots,n-1\}$.
\begin{thm}\label{thm: dom with small eigval multi holes}
Let $n\geq 2$ and $p\in \{1,\dots,n-1\}$. Given $\holes\in\NN$, there exists a family of Euclidean $C^1$\nobreakdash-domains with convex outer boundary and $\holes$ many holes, satisfying Hypothesis~\ref{hyp: geom condition on domain with holes} and having arbitrarily small exact eigenvalue, with the diameter of the domain uniformly bounded and the partition parameter tending to zero.
\end{thm}

\parnoi
We proceed with the proofs of the above theorems.

\parnoi
To prove Theorem~\ref{thm: mcgowan method bounds for multi holes}, we consider an open cover of $\dom$ given by the Laguerre-Voronoi partition and apply certain generalised versions of McGowan's gluing lemma (see Lemma~\ref{lem: various mcGow lemmas} and \hyperlink{appendix}{Appendix}) in combination with the lower bounds~\eqref{longer exprn of main lower bd p>=2}, \eqref{longer exprn of main thm p=1} of Theorem~\ref{thm: first abs eigval lower bds for forms on annular} applied to the cells.

\parnoi
\textbf{Proof of Theorem~\ref{thm: mcgowan method bounds for multi holes}. }
\noi
The cells $\{\mathcal{C}_i\}$ from the Laguerre-Voronoi partition have trivial cohomologies for $p\in \{1,\dots,n-2\}$. Let $\{\celleps{i}\}$ be smooth extensions of $\{\mathcal{C}_i\}$ around the parts of their outer boundaries in the interior of $\dom$, small enough not to intersect the hole from any neighbouring cell, such that $\{\celleps{i}\cup \BB_{c_i,r_i}\}$ are still convex and contained in $\Dom$.
\parnoi
Observe that the diameters,
\eqn{
\diam\left(\celleps{i}\right)\leq \diam(\dom)=D,
}
and the contact radii
\eqn{
R_c\left(\celleps{i}\right) \geq R_c(\mathcal{C}_i)\geq \Rpart.
}
Then we have from~\eqref{longer exprn of main lower bd p>=2} that, for $n\geq 3$ and $p\in \{2,\dots n-1\}$,
\begin{align}
\exeigpo\left(\celleps{i}\right) &\geq K \frac{1}{\diam\left(\celleps{i}\right)^2}\left(\frac{R_c\left(\celleps{i}\right)}{\diam\left(\celleps{i}\right)} \right)^{2n^2-n-2}\min\left\{1, \left(\frac{R_c\left(\celleps{i}\right)}{r_i}\right)^{3(n-1)(n+1)}\right\}\notag\\
&\geq K \frac{1}{D^2}\left(\frac{\Rpart}{D} \right)^{2n^2-n-2} \min\left\{1, \left(\frac{\Rpart}{\Rhh}\right)^{3(n-1)(n+1)}\right\},\qquad \forall \, i\in\{1,\dots,\holes\}.\label{lower bd for exeig of celleps p>=2}
\end{align}
Similarly, from~\eqref{longer exprn of main thm p=1} for $n\geq 2$ and $p=1$,
\eqnum{ \label{lower bd for exeig of celleps p=1}
\exeig{1}{1}\left(\celleps{i}\right) \geq K \frac{1}{D^2}\left(\frac{\Rpart}{D}\right)^{n-2}\min\left\{1,\left(\frac{\Rpart}{D}\right)^{3n+2}\left(\frac{\Rpart}{\Rhh}\right)^{7n-5}\right\},\qquad \forall \, i\in\{1,\dots,\holes\}.
}
Since $\{\celleps{i}\cup \BB_{c_i,r_i}\}$ are convex, so is any non-trivial intersection of a subset of them. From the known lower bound for convex domains \cite{Gue04i}, we have for any non-trivial intersection,
\eqnum{\label{lower bd for exeig of celleps intersections}
\exeig{p}{1}\left(\celleps{i_1}\cap\dots\cap\celleps{i_\ell}\right)\geq \frac{K}{\diam\left(\celleps{i_1}\cap\dots\cap\celleps{i_\ell}\right)^2}\geq \frac{K}{D^2},\qquad \forall \,p\in\{0,\dots,n\}.
}
\par
Below, we present certain variants of the Cheeger--McGowan lemma \cite[Lemma~2.3]{mcg93} and \cite[Lemma~1]{GP95}, which we use in this proof in combination with the above lower bounds for the individual cells.
\begin{lem}\label{lem: various mcGow lemmas}
Let $(M^n,g)$ be a compact Riemannian manifold with boundary and $\{U_i\}_{i=1}^{k_0}$ be an open cover with a partition of unity $\{\rho_i\}_{i=1}^{k_0}$ subordinate to it, with
\eqn{
c_\rho\df \max_{i\in \{1,\dots,k_0\}}\Vert \grad\rho_i\Vert^2_{\infty}.
}
Let $U_{i_0\dots i_\ell}$ denote an $\ell$\nobreakdash-intersection \mbox{$U_{i_0}\cap\dots\cap U_{i_\ell}$}, and $k_\ell$ the number of non-trivial $\ell$\nobreakdash-intersections. For convenience, we use the convention $\exeig{p}{1}(\varnothing)=\infty$. Then we have the following lower bounds for the exact eigenvalues of $M$ in terms of the first exact eigenvalues of $\{U_i\}$ and their intersections for $q$\nobreakdash-forms, $q\in\{0,\dots,p-1\}$.
\begin{enumerate}
\setlength{\itemsep}{1.5em}
\item For $n\geq 2$ and $p=1$,
\eqn{
\exeig{1}{1+k_1}(M)\geq \frac{1}{\sum_{i=1}^{k_0}\Inv{\exeig{1}{1}(U_i)}}.
}
\item Let $n\geq 3$ and $p=2$. Suppose that $H^1(U_{ij},\RR)=0$ for every non-trivial $1$\nobreakdash-intersection $U_{ij}$. Then
\eqn{
\exeig{2}{1+k_2}(M) \geq \frac{(8k_0)\inv}{\sum_{i=1}^{k_0}\Biggl\{\Inv{\exeig{2}{1}(U_i)}+\sum_{j=1}^{k_0}\Biggl[\left(\frac{c_\rho}{\exeig{1}{1}(U_{ij})}+1\right)\left(\Inv{\exeig{2}{1}(U_i)}+\Inv{\exeig{2}{1}(U_j)}\right)\Biggr]\Biggr\}}.
}
\item Let $n\geq 4$ and $p=3$. Suppose that $H^2(U_{ij},\RR)=0$ for every non-trivial $1$\nobreakdash-intersection $U_{ij}$ and $H^1(U_{ijk},\RR)=0$ for every non-trivial intersection $2$\nobreakdash-intersection $U_{ijk}$. Then
\eqn{
\exeig{3}{1+k_3}(M)\geq \frac{(18k_0^2)\inv}{\aleph},
}
where the denominator $\aleph$ is
\begin{samepage}
{\footnotesize
\begin{align*} \hspace*{3.4em}
\sum_{i=1}^{k_0}&\Biggl\{\Inv{\exeig{3}{1}(U_i)}+\sum_{j=1}^{k_0}\Biggl[\left(\frac{c_\rho}{\exeig{2}{1}(U_{ij})}+1\right)\left(\Inv{\exeig{3}{1}(U_i)}+\Inv{\exeig{3}{1}(U_j)}\right)\\
&+\sum_{k=1}^{k_0}c_\rho\left(\frac{c_\rho}{\exeig{1}{1}(U_{ijk})}+1\right)\biggl\{\frac{\Inv{\exeig{3}{1}(U_i)}+\Inv{\exeig{3}{1}(U_j)}}{\exeig{2}{1}(U_{ij})}+\frac{\Inv{\exeig{3}{1}(U_j)}+\Inv{\exeig{3}{1}(U_k)}}{\exeig{2}{1}(U_{jk})}+\frac{\Inv{\exeig{3}{1}(U_i)}+\Inv{\exeig{3}{1}(U_k)}}{\exeig{2}{1}(U_{ik})}\biggr\}\Biggr]\Biggr\}.
\end{align*}
}
\end{samepage}
\end{enumerate}
\end{lem}
\noi
Part~$(2)$ of this lemma for the case $n\geq 3$, $p=2$ is a direct consequence of \cite[Lemma~2.3]{mcg93}, where it was originally stated for closed manifolds but the proof also holds for compact manifolds with boundary. Part~$(1)$ for $n\geq 2$, $p=1$ is a straight forward adaptation of McGowan's method and is known in the literature, although there seems to be no  mention of this lower bound explicitly. The expression in the lower bound of Part~$(3)$ for $n\geq 4$, $p=3$ is a new result that we derive adapting McGowan's proof, motivated by the concluding remark of \cite[Section~2]{mcg93} that it should be possible to generalise their lemma to higher degree forms. The proof of this case is somewhat technical and is provided in the \hyperlink{appendix}{Appendix}.
\begin{remark} \label{rmk: mcGow for higher p}
Using similar arguments for higher degrees $p$, one can derive lower bounds on $\exeig{p}{1+k_p}(M)$ with more sophisticated expressions, supposing that the $q^{\text{th}}$ cohomology of every non-trivial $(p-q)$\nobreakdash-intersection is trivial for all $q\in \{1,\dots,p-1\}$. Specifically, the expression for a lower bound on $\exeig{p}{1+k_p}(M)$ will have a nested combinations of the inverses of $\exeig{p-q}{1}(U_{i_0\dots i_q})$ for $q\in\{0,\dots,p-1\}$, with a dependence of $k_0^{1-p}$ on the number of elements in the open cover of $M$, and the dependence on the partition of unity reflected in the factors 
\eqn{
\left(\frac{c_\rho}{\exeig{p-1}{1}(U_{i_0i_1})}+1\right),\,c_\rho\left(\frac{c_\rho}{\exeig{p-2}{1}(U_{i_0i_1i_2})}+1\right),\dots,\,c_\rho^{p-2}\left(\frac{c_\rho}{\exeig{p-3}{1}(U_{i_0i_1i_2i_3})}+1\right)
}
in the denominator.
\end{remark}
\par
Returning to the main proof of Theorem~\ref{thm: mcgowan method bounds for multi holes}, consider the open cover $\left\{\intr\celleps{i}\right\}_{i=1}^\holes$ of $\dom$ and a partition of unity $\{\rho_i\}_{i=1}^{\holes}$ subordinate to it such that
\eqn{
c_\rho\df \max_{i\in \{1,\dots,\holes\}}\Vert \grad\rho_i\Vert^2_{\infty}\leq \frac{K}{\Rpart^2}.
}
For $n\geq 2$ and $p=1$, it follows from part~$(1)$ of Lemma~\ref{lem: various mcGow lemmas} and the lower bound~\eqref{lower bd for exeig of celleps p=1} that
\eqn{
\exeig{1}{1+k_1}(\dom) \geq \frac{K}{\holes} \frac{1}{D^2}\left(\frac{\Rpart}{D}\right)^{n-2}\min\left\{1,\left(\frac{\Rpart}{D}\right)^{3n+2}\left(\frac{\Rpart}{\Rhh}\right)^{7n-5}\right\}.
}
For $n\geq 3$ and $p=2$, we use part~$(2)$ of Lemma~\ref{lem: various mcGow lemmas}, the lower bounds~\eqref{lower bd for exeig of celleps p>=2}~and~\eqref{lower bd for exeig of celleps intersections}, and that $c_\rho\leq K \Rpart^{-2}$, to get
\eqn{
\exeig{2}{1+k_2}(\dom)\geq \frac{K}{\holes^3}\frac{1}{D^2} \left(\frac{\Rpart}{D}\right)^{2n^2-n}\min\left\{1, \left(\frac{\Rpart}{\Rhh}\right)^{3(n-1)(n+1)}\right\}.
}
Likewise, for $n\geq 3$ and $p=3$, we use part~$(3)$ of Lemma~\ref{lem: various mcGow lemmas} to get
\eqn{
\exeig{3}{1+k_3}(\dom)\geq 
\frac{K}{\holes^5}\frac{1}{D^2} \left(\frac{\Rpart}{D}\right)^{2n^2-n+2}\min\left\{1, \left(\frac{\Rpart}{\Rhh}\right)^{3(n-1)(n+1)}\right\}.
}
For $n\geq 3$ and $p\in\{2,\dots,n-1\}$, in general, we use expression for higher degree forms described in Remark~\ref{rmk: mcGow for higher p}. Using the lower bounds~\eqref{lower bd for exeig of celleps p>=2}~and~\eqref{lower bd for exeig of celleps intersections}, and that 
\eqn{
c_\rho^{p-2}\leq K\max\{1,(\Rpart^{-2})^{p-2}\},
}
the denominator in the expression can be bounded above by
\eqn{
K \max\left\{1,(\Rpart^{-2})^{p-2}\right\}\left(\frac{\Rpart^{-2}}{D^{-2}}+1\right)\holes^p \left(\Inv{D^{-2}}\right)^{p-2} \left(\frac{1}{D^2}\left(\frac{\Rpart}{D} \right)^{2n^2-n-2} \min\left\{1, \left(\frac{\Rpart}{\Rhh}\right)^{3(n-1)(n+1)}\right\}\right)^{-1}
}
which yields
\begin{align*}
\exeig{p}{1+k_p}(\dom)&\geq \frac{K}{\holes^{2p-1}}\frac{\Rpart^2\min\left\{1,\Rpart^{2(p-2)}\right\}}{D^{2(p-1)}}\frac{1}{D^2} \left(\frac{\Rpart}{D}\right)^{2n^2-n-2}\min\left\{1, \left(\frac{\Rpart}{\Rhh}\right)^{3(n-1)(n+1)}\right\}\\
&\geq \frac{K}{\holes^{2p-1}}\frac{1}{D^2} \left(\frac{\Rpart}{D}\right)^{2n^2-n-2+2(p-1)}\min\left\{1,\Inv{\Rpart^{2(p-2)}}\right\}\min\left\{1, \left(\frac{\Rpart}{\Rhh}\right)^{3(n-1)(n+1)}\right\}\\
&\geq \frac{K}{\holes^{2p-1}}\frac{1}{D^2} \left(\frac{\Rpart}{D}\right)^{2n^2+n-6}\min\left\{1,\Inv{\Rpart^{2(n-3)}}\right\}\min\left\{1, \left(\frac{\Rpart}{\Rhh}\right)^{3(n-1)(n+1)}\right\},
\end{align*}
giving the required lower bound, using that $\Rhh \le D$.
\qed
\par
In order to obtain lower bounds for the first exact eigenvalue, we instead use directly Lemma~\ref{lem: modified McGowan} by starting with a cell in the open cover and adding cells inductively such that the convexity of the outer boundary of the intermediate domain is preserved at each step (there exists such an ordering guaranteed by Hypothesis~\ref{hyp: geom condition on domain with holes}), and invoking once again the lower bounds~\eqref{longer exprn of main lower bd p>=2}, \eqref{longer exprn of main thm p=1} of Theorem~\ref{thm: first abs eigval lower bds for forms on annular} for the cells being added. This procedure introduces the dependence of the exponents of the ratio $\Rpart/D$ on the number of holes. As before, for $p=1$, the gluing lemma needs to be replaced by Theorem~\ref{thm: neumann lower bd for func in terms of union} and we adapt our approach accordingly, as described in the proof below.
\parnoi
\textbf{Proof of Theorem~\ref{thm: induction method bounds for multi holes}. }As before, consider the open cover $\{\intr\celleps{i}\}_{i=1}^\holes$ of $\dom$, where $\{\mathcal{C}_i\}_{i=1}^\holes$ are ordered as given by Hypothesis~\ref{hyp: geom condition on domain with holes}, and such that $\bigcup_{i=1}^\ell \,\celleps{i} \cup \, \BB_{c_i,r_i}$ is convex for each \mbox{$\ell \in \{1,\dots,\holes\}$}. Let $\{\rho_i\}_{i=1}^{\holes}$ be a partition of unity subordinate to this open cover, satisfying
\eqn{
c_\rho\df \max_{i\in \{1,\dots,\holes\}}\Vert \grad\rho_i\Vert^2_{\infty}\leq \frac{K}{\Rpart^2}.
}
\parnoi
When $n\geq 3$ and $p\in\{2,\dots,n-1\}$, we proceed to prove the following using induction: For each $\ell\in\{1,\dots,\holes\}$,
\eqn{
\exeigpo\left(\celleps{1}\cup\dots\cup\celleps{\ell}\right)\geq K \frac{1}{D^2} \left(\frac{\Rpart}{D}\right)^{2n^2-n-2+2(\ell-1)}\min\left\{1, \left(\frac{\Rpart}{\Rhh}\right)^{3(n-1)(n+1)}\right\}.
}
The above statement is true for $\ell=1$ from the lower bound~\eqref{lower bd for exeig of celleps p>=2}. Suppose that the induction hypothesis holds for some $\ell \in \{1,\dots,\holes-1\}$.  Using Lemma~\ref{lem: modified McGowan} for the open cover of $\{\celleps{1}\cup\dots\cup\celleps{\ell+1}\}$ given by $\{\celleps{1}\cup\dots\cup\celleps{\ell}, \celleps{\ell+1}\}$, and that $c_\rho\leq K R_c^{-2}$, we get
\eqn{
\exeigpo\left(\celleps{1}\cup\dots\cup\celleps{\ell+1}\right)\geq K \frac{1}{D^2} \left(\frac{\Rpart}{D}\right)^{2n^2-n-2+2\ell}\min\left\{1, \left(\frac{\Rpart}{\Rhh}\right)^{3(n-1)(n+1)}\right\},
}
hence the induction statement is true for $\ell+1$. The first part of Theorem~\ref{thm: induction method bounds for multi holes} follows.
\parnoi
When $n\geq 2$ and $p=1$, we proceed as follows, again using induction: For each $\ell\in\{1,\dots,\holes\}$,
\eqn{
\exeig{1}{1}\left(\celleps{1}\cup\dots\cup\celleps{\ell}\right)\geq K \frac{1}{D^2} \left(\frac{\Rpart\rhh^{n-1}}{D^n}\right)^{\ell-1}\left(\frac{\Rpart}{D}\right)^{n-2}\min\left\{1,\left(\frac{\Rpart}{D}\right)^{3n+2}\left(\frac{\Rpart}{\Rhh}\right)^{7n-5}\right\}.
}
The above statement is true for $\ell=1$ from the lower bound~\eqref{lower bd for exeig of celleps p=1}. Suppose that the induction hypothesis holds for some $\ell \in \{1,\dots,\holes-1\}$. Using Theorem~\ref{thm: neumann lower bd for func in terms of union},
\eqn{
\exeig{1}{1}\left(\celleps{1}\cup\dots\cup\celleps{\ell+1}\right)\geq \frac{1}{32} \frac{\vol\left((\celleps{1}\cup\dots\cup\celleps{\ell})\cap \celleps{\ell+1}\right)}{\vol\left(\celleps{1}\cup\dots\cup\celleps{\ell+1}\right)}\min\left\{\exeig{1}{1}\left(\celleps{1}\cup\dots\cup\celleps{\ell}\right),\exeig{1}{1}\left(\celleps{\ell+1}\right)\right\}.
}
We have that
\eqn{
\vol\left(\celleps{1}\cup\dots\cup\celleps{\ell+1}\right)\leq K D^n,
}
while the volume term in the numerator may be estimated as
\eqn{
\vol\left((\celleps{1}\cup\dots\cup\celleps{\ell})\cap \celleps{\ell+1}\right) \geq K \Rpart r_{\ell+1}^{n-1} \geq K \Rpart\rhh^{n-1},
}
since the extensions $\celleps{1},\dots,\celleps{\ell}$ may be chosen large enough such that their union contains at least the subset of $\mathcal{C}_{\ell+1}$ given by the cylinder aligned along the perpendicular from $c_{\ell+1}$ to the hyperplane of partition, of radius equal to $r_{\ell+1}$ and height $\Rpart$, with its centre at a distance of $r_{\ell+1}+\frac{\Rpart}{2}$ from $c_{\ell+1}$. Using these volume estimates along with the induction hypothesis and the lower bound~\eqref{lower bd for exeig of celleps p=1}, we get
\eqn{
\exeig{1}{1}\left(\celleps{1}\cup\dots\cup\celleps{\ell+1}\right)\geq K \frac{1}{D^2} \left(\frac{\Rpart\rhh^{n-1}}{D^n}\right)^{\ell}\left(\frac{\Rpart}{D}\right)^{n-2}\min\left\{1,\left(\frac{\Rpart}{D}\right)^{3n+2}\left(\frac{\Rpart}{\Rhh}\right)^{7n-5}\right\},
}
hence the induction statement is true for $\ell+1$. The second part of Theorem~\ref{thm: induction method bounds for multi holes} follows.
\qed
\parnoi
\textbf{Proof of Theorem~\ref{thm: dom with small eigval multi holes}. }Fix $n\geq 2$, $p\in\{1, \dots, n-1\}$, and $\holes \in \NN$. Consider the family of domains $\{\ann^{p-1}_\eps\}_{\eps\in[0,\half)}$ as in the proof of Theorem~\ref{thm: domains with small eigval} in  Section~\ref{sec: small eigvals}, and the cartesian coordinates $(x_1,\dots,x_n)$ of $\RR^n=\RR^{n-p}\times\RR^p$, where $\RR^{n-p}$ and $\RR^p$ are as before. Recall that Section~\ref{sec: small eigvals} approaches the proof of Theorem~\ref{thm: domains with small eigval} using the corresponding coexact eigenvalues for convenience, whereas we now directly analyse the exact eigenvalues of degree $p$, hence the choice $\{\ann^{p-1}_\eps\}$ associated with $p$. For sufficiently small $r_0>0$ depending on $\holes$, construct the family of domains $\{\ann^{p-1}_{\eps,r}\}$ from $\{\ann^{p-1}_\eps\}$ by removing $\holes-1$ non-overlapping balls of radius $r\in(0,r_0)$ contained in the interior of $\ann^{p-1}_0$ and centered on the $x_1$\nobreakdash-axis.
\begin{figure}[H]
    \centering
    \includegraphics[width=0.6\linewidth]{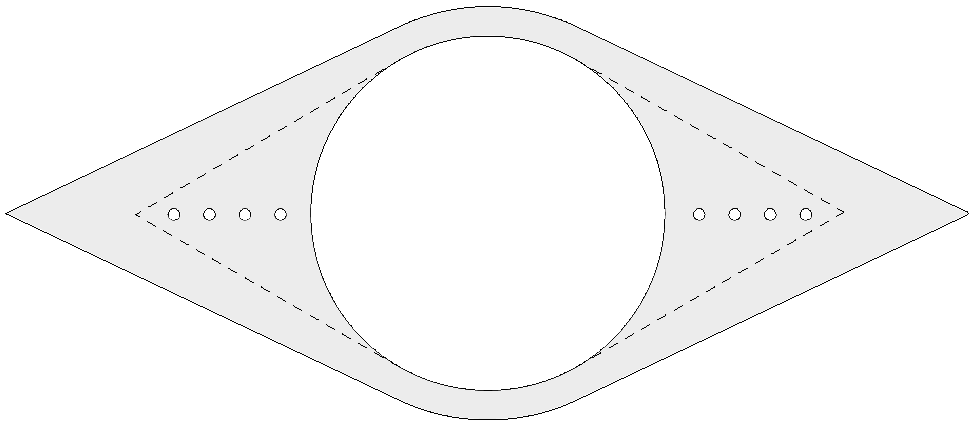}
    \caption{}
\end{figure}
\parnoi
Clearly, the domains $\{\ann^{p-1}_{\eps,r}\}$ satisfy Hypothesis~\ref{hyp: geom condition on domain with holes}, and have uniformly bounded diameter and radius of the largest hole, with the partition parameter tending to zero as $\eps\to 0$. It then follows from Proposition~\ref{prop: anne colbois spectral convergence adaptation} and Theorem~\ref{thm: domains with small eigval} that
\eqn{
\lim_{\eps\to 0}\lim_{r\to 0} \exeigpo (\ann^{p-1}_{\eps,r})=\lim_{\eps\to 0} \exeigpo (\ann^{p-1}_\eps)=\lim_{\eps\to 0} \coeig{p-1}{1}(\ann^{p-1}_\eps)=0,
}
thereby giving rise to arbitrarily small exact eigenvalues.
\qed

\hypertarget{appendix}{}
\section*{Appendix: Generalisation of McGowan's lemma}
We give a proof of part~$(3)$ of  Lemma~\ref{lem: various mcGow lemmas}, adapting that of~\cite[Lemma~2.3]{mcg93} to the case of $n\geq 4$, $p=3$. We begin by recalling a variational characterisation for the exact eigenvalues with absolute boundary conditions~\cite[Proposition~3.1]{dod82} (see \cite[Proposition~2.1]{mcg93} for a proof): For $p\in \{1,\dots,n\}$,
\eqnum{\label{varchar of dodziuk}
\exeig{p}{k}(M)=\inf_{V_k} \sup_{\eta\in V_k\setminus \{0\}}\, \sup_{\theta\,:\, \eta=d\theta} \frac{\Vert \eta \Vert^2_{L^2\Om^p(M)}}{\Vert \theta \Vert^2_{L^2\Om^{p-1}(M)}},
}
where $V_k$ ranges over all $k$\nobreakdash-dimensional subspaces of the space of $L^2$\nobreakdash-integrable exact $p$\nobreakdash-forms on $M$. Let $\mathcal{E}_k$ denote the span of the eigenforms corresponding to $\left\{\exeig{3}{1},\dots,\exeig{3}{k}\right\}$. Then for any $\phi \in \mathcal{E}_{1+k_3}$ and $\psi$ such that $d\psi=\phi$,
\eqn{
\exeig{3}{1+k_3}\geq \frac{\Vert \phi \Vert^2_{L^2\Om^3(M)}}{\Vert \psi \Vert^2_{L^2\Om^2(M)}}.
}
So, it will suffice to construct such a pair with the above quotient of the (squares of the) $L^2$\nobreakdash-norms bounded below by appropriate terms. Given a $\phi$, we can locally construct $\psi$ using the eigenvalues of the pieces $\{U_i\}$ and their intersections. We then use a partition of unity for the local-to-global transition employing the \v{C}ech--de Rham formalism, which enforces a choice of $\phi$ to be made in the proof, as we shall see.
\begin{figure}[H]
\centering
\begin{tikzcd}
&\vdots&\vdots&\vdots&\\
0 \arrow[r] &\Om^2(M)\arrow[u,"d"]\arrow[r,"\restr"] &\prod\Om^2(U_i)\arrow[u,"d"]\arrow[r,"\diff"]&\prod\Om^2(U_{ij})\arrow[u,"d"]\arrow[r,"\diff"]&\dots\\
0 \arrow[r] &\Om^1(M)\arrow[u,"d"] \arrow[r,"\restr"] &\prod\Om^1(U_i)\arrow[u,"d"]\arrow[r,"\diff"]&\prod\Om^1(U_{ij})\arrow[u,"d"]\arrow[r,"\diff"]&\dots\\
0 \arrow[r] &\Om^0(M)\arrow[u,"d"] \arrow[r,"\restr"] &\prod\Om^0(U_i)\arrow[u,"d"]\arrow[r,"\diff"]&\prod\Om^0(U_{ij})\arrow[u,"d"]\arrow[r,"\diff"]&\dots\\
&&\check{C}^0(\{U_i\},\RR)\arrow[u,"\iota"]\arrow[r]&\check{C}^1(\{U_i\},\RR)\arrow[u,"\iota"]\arrow[r]&\dots\\
&&0\arrow[u]&0\arrow[u]
\end{tikzcd}
\caption{\v{C}ech--de Rham formalism}
\label{fig: cech-deRham}
\end{figure}
\parnoi
For $\om\in \Om^p(M)$, we denote by $\om_i$ the restriction of $\om$ to $U_i$. Let $\restr:\Om^p(M)\to \prod_{i=1}^{k_0}\Om^p(U_i)$ denote the restriction map, and $\diff$
denote the difference operator as defined in~\cite[Section~2.3]{mcg93}:
\begin{align*}
\diff:\prod\Om^p(U_{i_0\dots i_m})&\to \prod\Om^p(U_{i_0\dots i_{m+1}})\\
(\diff \om)_{i_0\dots i_{m+1}}&\df\sum_{s=0}^{m+1}(-1)^s\, \om_{i_0\dots \hat{i}_s \dots i_{m+1}},
\end{align*}
where the caret in the subscript means omission, and we only consider non-trivial intersections of orders $m$ and $m+1$ in the above definition. The rows in the diagram of Figure~\ref{fig: cech-deRham} are exact, except the last row
\eqn{
\check{C}^0(\{U_i\},\RR) \xrightarrow{\quad}\check{C}^1(\{U_i\},\RR)\xrightarrow{\quad}\check{C}^2(\{U_i\},\RR)\xrightarrow{\quad}\dots,
}
with $\check{C}^m(\{U_i\},\RR)$ denoting the space of locally constant functions on non-trivial $m$\nobreakdash-intersections of $\{U_i\}$. This sequence is a differential complex, and its homology is called the \emph{\v{C}ech cohomology} of the cover $\{U_i\}_{i=1}^{k_0}$.
See~\cite[Chapter~8]{bott-tu} for a detailed description of the \v{C}ech-de~Rham formalism.
\parnoi
We start with a $3$\nobreakdash-form $\phi\in \mathcal{E}_{1+k_3}$, the choice of which would be described later in this proof. As its restrictions $\{\phi_i\}_{i=1}^{k_0}$ are exact on $\{U_i\}_{i=1}^{k_0}$, there exist unique coexact $2$\nobreakdash-forms $\{\psi_i\}_{i=1}^{k_0}$ (with appropriate boundary conditions described in the Hodge-decomposition~\eqref{hodge decomp from taylor pde}) such that $d\psi_i=\phi_i$ for each $i\in\{1,\dots,k_0\}$. Since $\psi_i$ minimises the $L^2$\nobreakdash-norm among all the primitives of $\phi_i$ on $U_i$, the variational characterisation~\eqref{varchar of dodziuk} gives
\eqnum{\label{exeig3-1 ineq}
\exeig{3}{1}(U_i)\leq \frac{\Vert \phi_i \Vert^2_{L^2\Om^3(M)}}{\Vert \psi_i \Vert^2_{L^2\Om^2(M)}}, \qquad \forall \,i\in \{1,\dots,k_0\}.
}
Consider the $2$\nobreakdash-forms $\{\om_{ij}\}\df \diff\{\phi_i\}$ on non-trivial $1$\nobreakdash-intersections $\{U_{ij}\}$, i.e., $\om_{ij}=\psi_j-\psi_i$. They are closed because
\eqn{
d\om_{ij}=d\psi_j-d\psi_i=\phi-\phi=0,
}
and since $H^2(U_{ij},\RR)=0$ for all $U_{ij}$ by hypothesis, there exist unique coexact $1$\nobreakdash-forms $\{\eta_{ij}\}$ (with appropriate boundary conditions) on $\{U_{ij}\}$ such that $\om_{ij}=d\eta_{ij}$. Like before, we have from the variational characterisation~\eqref{varchar of dodziuk} that
\eqnum{\label{exeig2-1 ineq}
\exeig{2}{1}(U_{ij})\leq \frac{\Vert \om_{ij} \Vert^2_{L^2\Om^2(M)}}{\Vert \eta_{ij} \Vert^2_{L^2\Om^1(M)}},
}
for every non-trivial $1$\nobreakdash-intersection $U_{ij}$. Now, consider the $1$\nobreakdash-forms $\{\nu_{ijk}\}\df\diff \{\eta_{ij}\}$ on non-trivial $2$\nobreakdash-intersections $\{U_{ijk}\}$. Since $d$, $\diff$ commute and $\diff^2=0$, the $\{\nu_{ijk}\}$ are closed:
\eqn{
d\{\nu_{ijk}\}=d\diff \{\eta_{ij}\}=\diff d\{\eta_{ij}\}=\diff\{\om_{ij}\} =\diff^2 \{\phi_i\}=0.
}
Again, since $H^1(U_{ijk},\RR)=0$ for all $U_{ijk}$, there exist unique coexact functions $\{\nuh_{ijk}\}$ on $\{U_{ijk}\}$ such that $\nu_{ijk}=d\nuh_{ijk}$, and the variational characterisation gives
\eqnum{\label{exeig1-1 ineq}
\exeig{1}{1}(U_{ijk})\leq \frac{\Vert \nu_{ijk} \Vert^2_{L^2\Om^1(M)}}{\Vert \nuh_{ijk} \Vert^2_{L^2\Om^0(M)}}.
}
\begin{figure}[H]
\centering
\begin{tikzcd}[column sep={4.8em,between origins}, row sep={4em,between origins}]
\text{$3$\nobreakdash-forms} &\{\phi_i\}  &0 &&&&\{\phi_i\}  &0\\
\text{$2$\nobreakdash-forms} &\{\psi_i\} \arrow[u,"d"] \arrow[r,"\diff"] &\{\om_{ij}\} \arrow[u,"d"] \arrow[r,"\diff"] &0 &&&\{\psi_i\} \arrow[u,"d"] \arrow[r,"\diff"] &\{\om_{ij}\} \arrow[u,"d"] \arrow[r,"\diff"] &0 \\
\text{$1$\nobreakdash-forms} & &\{\eta_{ij}\} \arrow[u,"d"] \arrow[r,"\diff"] &\{\nu_{ijk}\} \arrow[u,"d"] \arrow[r,"\diff"] &0 &&\{\tau_i\}\arrow[r,"\diff"] &\{\etab_{ij}\}\arrow[u,"d"] \arrow[r,"\diff"] &0 \\
\text{Functions} & & \{\etah_{ij}\}\arrow[r,dashed,"\diff"]& \{\nuh_{ijk}\}\arrow[u,"d"]\arrow[r,"\diff"]&\{\beta_{ijkl}\}\arrow[u,"d"]
\end{tikzcd}
\end{figure}
\parnoi
Had there been no non-trivial $3$\nobreakdash-intersections $\{U_{ijkl}\}$ (i.e., $k_3=0$), it would follow from the exactness of the generalised Mayer-Vietoris sequence~\cite[Proposition~8.5]{bott-tu} that there exist $\{\etah_{ij}\}$ such that $\diff\{\etah_{ij}\}=\{\nuh_{ijk}\}$. When $k_3\neq0$, one can still argue that, since the dimension of $\mathcal{E}_{1+k_3}$ is strictly greater than $k_3$, there exists a choice of $\phi\in \mathcal{E}_{1+k_3}$ such that $\{\beta_{ijkl}\}\df \diff\{\nuh_{ijk}\}=0$, which guarantees from the exactness that there exist $\{\etah_{ij}\}$ such that $\diff\{\etah_{ij}\}=\{\nuh_{ijk}\}$. In particular, we can verify that
\eqn{
\etah_{ij}=\sum_{l=1}^{k_0} \rho_l\, \nuh_{lij}
}
would satisfy the required condition. The expression on the RHS is from a homotopy operator~$K$ (see~\cite[Chapter~8]{bott-tu}) acting on $\{\nuh_{ijk}\}$, whose property that \hbox{$K\diff+\diff K= 1$} implies $\diff\{\etah_{ij}\}=\diff K \{\nuh_{ijk}\}=\{\nuh_{ijk}\}$ since $\diff\{\nuh_{ijk}\}=0$ for our choice of $\phi$. We then set
\eqn{
\{\etab_{ij}\}\df \{\eta_{ij}\}-d\{\etah_{ij}\},
}
whence we have
\eqn{
d\{\etab_{ij}\}=d\{\eta_{ij}\}-d^2\{\etah_{ij}\}=d\{\eta_{ij}\}
}
and
\begin{align*}
\diff \{\etab_{ij}\}&=\diff\{\eta_{ij}\}-\diff d\{\etah_{ij}\}\\
&=\diff\{\eta_{ij}\}-d\diff \{\etah_{ij}\}\\
&=\diff\{\eta_{ij}\}-d\{\nuh_{ijk}\}\\
&=\diff\{\eta_{ij}\}-\{\nu_{ijk}\}\\
&=0.
\end{align*}
Once again, the exactness of the generalised Mayer-Vietoris sequence provides with $\{\tau_i\}$ such that $\diff \{\tau_i\}=\{\etab_{ij}\}$, also explicitly given by
\eqn{
\tau_i\df\sum_j \rho_j\,\etab_{ij}.
}
Finally, we define $\{\psib_i\}\df\{\psi_i\}-d\{\tau_i\}$, which satisfy
\eqn{
d\{\psib_i\}=d\{\psi_i\}-d^2\{\tau_i\}=d\{\psi_i\}
}
and
\begin{align*}
\diff \{\psib_i\}&=\diff\{\psi_i\}-\diff d\{\tau_i\}\\
&=\diff\{\psi_i\}-d \diff \{\tau_i\}\\
&=\diff\{\psi_i\}-d \{\etab_{ij}\}\\
&=\diff\{\psi_i\}-\{\om_{ij}\}\\
&=0.
\end{align*}
In other words, $\{\psib_i\}$ agree on all the non-trivial intersections $\{U_{ij}\}$ and thus give rise to a global $2$\nobreakdash-form on $M$, say $\psib$.
\parnoi
We now proceed to derive the required lower bound for the quotient
\eqn{
\frac{\Vert \phi \Vert^2_{L^2\Om^3(M)}}{\Vert \psib \Vert^2_{L^2\Om^2(M)}},
}
using the upper bounds for $\exeig{3}{1}(U_i)$, $\exeig{2}{1}(U_{ij})$ and $\exeig{1}{1}(U_{ijk})$ obtained so far. Firstly,
\begin{align*}
\psib_i&=\psi_i-d\tau_i\\
&=\psi_i-d\sum_{j=1}^{k_0}\rho_j\etab_{ij}\\
&=\psi_i-d\sum_{j=1}^{k_0}\rho_j(\eta_{ij}-d\etah_{ij})\\
&=\psi_i-d\sum_{j=1}^{k_0}\rho_j\left(\eta_{ij}-d\sum_{k=1}^{k_0}\rho_k\nuh_{kij}\right).
\end{align*}
We have from~\eqref{exeig3-1 ineq} that
\eqn{
\nrm{\psi_i}^2\leq \frac{\nrm{\phi_i}^2}{\exeig{3}{1}(U_i)} \leq \frac{\nrm{\phi}^2}{\exeig{3}{1}(U_i)}.
}
Also,
\begin{align*}
\nrm{d\sum_{j=1}^{k_0}\rho_j \eta_{ij}}^2&=\nrm{\sum_{j=1}^{k_0} d\rho_j\wedge \eta_{ij}\pm\sum_{j=1}^{k_0}\rho_j \,d\eta_{ij}}^2\\
&\leq 2k_0 \left[\sum_{j=1}^{k_0}\nrm{d\rho_j\wedge \eta_{ij}}^2 +\sum_{j=1}^{k_0}\nrm{\rho_j\,d\eta_{ij}}^2 \right]\\
&\leq 2k_0\left[\sum_{j=1}^{k_0}c_\rho \nrm{\eta_{ij}}^2+\sum_{j=1}^{k_0}\nrm{d\eta_{ij}}^2 \right]\\
&\overset{\text{from~\eqref{exeig2-1 ineq}}}{\leq} 2k_0\sum_{j=1}^{k_0}\left(\frac{c_\rho}{\exeig{2}{1}(U_{ij})}+1  \right)\nrm{d\eta_{ij}}^2\\
&=2k_0\sum_{j=1}^{k_0}\left(\frac{c_\rho}{\exeig{2}{1}(U_{ij})}+1  \right)\nrm{\psi_i-\psi_j}^2\\
&\leq 4k_0 \sum_{j=1}^{k_0}\left(\frac{c_\rho}{\exeig{2}{1}(U_{ij})}+1  \right)\left(\nrm{\psi_i}^2+\nrm{\psi_j}^2\right),
\end{align*}
and
\begin{align*}
\nrm{\sum_{j=1}^{k_0} d\rho_j \wedge \left(\sum_{k=1}^{k_0}d(\rho_k \nuh_{kij})\right)}^2&\leq k_0 \sum_{j=1}^{k_0} \nrm{d\rho_j \wedge \left(\sum_{k=1}^{k_0}d(\rho_k \nuh_{kij})\right)}^2\\
&\leq k_0 c_\rho \sum_{j=1}^{k_0} \nrm{\sum_{k=1}^{k_0}d(\rho_k \nuh_{kij})}^2\\
&\leq k_0^2 c_\rho \sum_{j=1}^{k_0} \sum_{k=1}^{k_0}\nrm{d(\rho_k \nuh_{kij})}^2,
\end{align*}
with
\begin{align*}
\nrm{d(\rho_k \nuh_{kij})}^2&=\nrm{d\rho_k\,\nuh_{kij}\pm\rho_k\,d\nuh_{kij}}^2\\
&\leq 2\left[c_\rho \nrm{\nuh_{ijk}}^2 +\nrm{d\nuh_{ijk}}^2\right]\\
&\overset{\text{from~\eqref{exeig1-1 ineq}}}{\leq}2\left(\frac{c_\rho}{\exeig{1}{1}(U_{kij})}+1\right)\nrm{\nu_{kij}}^2\\
&\leq 2\left(\frac{c_\rho}{\exeig{1}{1}(U_{kij})}+1\right) \nrm{\eta_{ij}\pm\eta_{ik}\pm\eta_{jk}}^2\\
&\leq 6\left(\frac{c_\rho}{\exeig{1}{1}(U_{kij})}+1\right) \left(\nrm{\eta_{ij}}^2+\nrm{\eta_{ik}}^2+\nrm{\eta_{jk}}^2\right)\\
&\leq 6\left(\frac{c_\rho}{\exeig{1}{1}(U_{kij})}+1\right) \left(\frac{\nrm{\psi_i}^2+\nrm{\psi_j}^2}{\exeig{2}{1}(U_{ij})} + \frac{\nrm{\psi_i}^2+\nrm{\psi_k}^2}{\exeig{2}{1}(U_{ik})} + \frac{\nrm{\psi_j}^2+\nrm{\psi_k}^2}{\exeig{2}{1}(U_{jk})}  \right).
\end{align*}
Using all the above estimates, we get
\begin{multline*}
\nrm{\psib_i}^2 \leq 3\nrm{\psi_i}^2 + 12k_0\sum_{j=1}^{k_0}\left(\frac{c_\rho}{\exeig{2}{1}(U_{ij})}+1  \right)\left(\nrm{\psi_i}^2+\nrm{\psi_j}^2\right) \\  + 18k_0^2c_\rho\sum_{j=1}^{k_0}\sum_{k=1}^{k_0}\left(\frac{c_\rho}{\exeig{1}{1}(U_{kij})}+1\right) \left(\frac{\nrm{\psi_i}^2+\nrm{\psi_j}^2}{\exeig{2}{1}(U_{ij})} + \frac{\nrm{\psi_i}^2+\nrm{\psi_k}^2}{\exeig{2}{1}(U_{ik})} + \frac{\nrm{\psi_j}^2+\nrm{\psi_k}^2}{\exeig{2}{1}(U_{jk})}  \right).
\end{multline*}
Employing inequality~\eqref{exeig3-1 ineq} and taking the summation on both sides with $i$ in the range $\{1,\dots,k_0\}$, we get
\eqn{
\frac{\nrm{\psib}^2}{\nrm{\phi}^2}\leq \frac{\sum_{i=0}^{k_0}\nrm{\psib_i}^2}{\nrm{\phi}^2}\leq 18k_0^2 \aleph,
}
where $\aleph$ is as defined in the statement of Lemma~\ref{lem: various mcGow lemmas}. Finally, the variational characterisation~\eqref{varchar of dodziuk} implies that
\eqn{
\exeig{3}{1+k_3}(M)\geq \frac{\Vert \phi \Vert^2_{L^2\Om^3(M)}}{\Vert \psib \Vert^2_{L^2\Om^2(M)}} \geq \frac{(18k_0^2)\inv}{\aleph}.
}\qed
\section*{Acknowledgement}
The authors are grateful to Bruno Colbois for sharing his valuable insights and feedback on the preprint. Tirumala Chakradhar is a PhD student at the University of Bristol and thanks his supervisor Asma Hassannezhad for her guidance and the numerous discussions. The authors would also like to thank Alessandro Savo for the helpful feedback and suggestions, and the anonymous reviewer whose comments have improved the presentation.

\end{document}